\documentclass[12pt]{article}

\usepackage{latexsym,amsfonts,amssymb,epsfig,verbatim,}
\usepackage{amsmath,amsthm,amssymb,latexsym,graphics,textcomp,authblk}
\usepackage{graphicx}

\usepackage{url}

\newtheorem{thm}{Theorem}[section]

\newtheorem{lem}[thm]{Lemma}
\newtheorem{prop}[thm]{Proposition}

\newcommand{\BS}{\mathbb S}

\newcommand{\BZ}{\mathbb Z}

\newcommand{\BF}{\mathbb F}

\DeclareMathOperator{\Sym}{Sym}
\DeclareMathOperator{\Mod}{Mod}
\DeclareMathOperator{\dist}{dist}

\DeclareMathOperator{\Homeo}{Homeo}
\DeclareMathOperator{\Aut}{Aut}
\DeclareMathOperator{\st}{st}
\DeclareMathOperator{\sth}{\widehat{st}}

\def\id{\mathrm{id}}

\newcommand{\mathsym}[1]{{}}
\newcommand{\unicode}[1]{{}}

\begin{document}

\title{Finite rigid subgraphs of the pants graphs of punctured spheres} 

\author{Rasimate Maungchang}

\maketitle

\begin{abstract}
We prove a strong form of finite rigidity for pants graphs of spheres. Specifically, for any $n\geq 4$, we construct a finite subgraph $X_n$ of the pants graph $\mathcal{P}(S_{0,n})$ of the $n$-punctured sphere $S_{0,n}$ with the following property. Any simplicial embedding of $X_n$ into any pants graph $\mathcal{P}(S_{0,m})$ of a punctured sphere is induced by an embedding $S_{0,n}\to S_{0,m}$.
\end{abstract}
\section{Introduction}
\label{sec:Introduction}
Let $S=S_{g,n}$ be an orientable surface of genus $g$ with $n$ punctures. The pants graph $\mathcal{P}(S)$ of $S$ has vertices corresponding to  pants decompositions of $S$ and edges corresponding to elementary moves (see Section~\ref{sec:background} for details).

Margalit \cite{Mar} proved that, for most surfaces $S$, $\Aut(\mathcal{P}(S))\cong \Mod^{\pm}(S)$, where $\Mod^{\pm}(S)=\pi_0(\Homeo(S))$ is the extended mapping class group.
This result was extended by Aramayona \cite{Aramayona} who proved that, for any two surfaces $S$ and $S'$ such that the complexity of $S$ is at least $2$, every injective simplicial map $\phi:\mathcal{P}(S)\to \mathcal{P}(S')$ is induced by a $\pi_1$-injective embedding $f:S\to S'$. For related results on curve complexes see ~\cite{Ivanov,Korkmaz,Luo,SK,AL}.

$\mathcal{P}(S)$ is an infinite and locally infinite graph. In this paper, we refine Aramayona's result (for punctured spheres) and prove the following.
\begin{thm}
\label{T:main theorem}
For $n\geq 4$, there exists a finite subgraph $X_n \subset \mathcal{P}(S_{0,n})$ such that for any punctured sphere $S_{0,m}$ and any injective simplicial map 
\[
\phi:X_n \to \mathcal{P}(S_{0,m}),
\]
there exists a $\pi_1$-injective embedding $f:S_{0,n} \to S_{0,m}$ that induces $\phi$. 

For $n=4$, the isotopy class of $f$ is unique up to precomposing by an element $\sigma\in\Mod(S_{0,4})$ inducing the identity on $\mathcal{P}(S_{0,4})$.

For $n\geq 5$, the isotopy class of $f$ is unique.
\end{thm}
We say that $f$ \textbf{induces} $\phi$ if there is a deficiency-$(n-3)$ multicurve $Q$ on $S_{0,m}$ with the following property  (see Section~\ref{sec:background} for definitions). The image $f(S_{0,n})$ is the unique non-pants component $(S_{0,m}-Q)_0\subset S_{0,m}-Q$ and the simplicial map  
\[
f^{Q}:\mathcal{P}(S_{0,n}) \to \mathcal{P}(S_{0,m}), 
\]
defined by
$
f^{Q}(u)=f(u)\cup Q
$
satisfies $f^{Q}(u)=\phi(u)$ for any vertex $u\in X_n$.

This result is analogous to Aramayona-Leininger~\cite{AL} for the case of the curve complex when $n=m$ (that theorem applies to arbitrary surfaces).

One of the main difficulties in proving Theorem~\ref{T:main theorem} is to construct the finite subgraphs $X_n$. To do this we can look for a candidate subgraph which, under additional hypotheses on the simplicial map, allows us to construct the embedding of the surface. For example, we have the condition on the simplicial map of $Z_5$ in Lemma~\ref{lem:pentagon}. We then need to enlarge the candidate subgraph so that those extra conditions are encoded in the enlarged subgraph. But then another problem might arise which is that the induced map of the embedding that works on the original candidate subgraph might not control the added parts in the enlarged subgraph. For example, a simplicial embedding of the thick pentagon $\widehat{Z_5}$ ensures that the simplicial map restricted to $Z_5$ satisfies Lemma~\ref{lem:pentagon} and hence there is a candidate embedding of $S_{0,5}$. But the induced map may not agree with the simplicial map on $\widehat{Z_5}-Z_5$. Then we have to enlarge the subgraph further which might cause more problems.

It seems likely that Theorem~\ref{T:main theorem} should be true for essentially any surface $S$. However it is unclear how to choose a subgraph $X\subset P(S)$.

{\bf Outline of the paper.} Section~\ref{sec:background} contains the relevant background and definitions. This section also contains the proof of the theorem in the case of $n=4$.
In Section~\ref{sec:X_n}, we describe the finite subgraph $X_n$ and prove some important properties of the subgraph. 
We prove the theorem for the case of $n=5$ in Section~\ref{sec:base case} and for the general case in Section~\ref{sec:general case}.

{\bf Acknowledgments.}
I would like to thank my advisor Christopher J. Leininger for guidance and useful conversations. I also would like to thank Javier Aramayona, Dan Margalit, and the referees for their helpful comments.
\section{Background and definitions}
\label{sec:background}
Here we describe some of the background material. See \cite{Aramayona} and \cite{Mar} for more details. Let $S=S_{g,n}$ be an orientable surface of genus $g$ with $n$ punctures. A simple closed curve on $S$ is \textbf{essential} if it does not bound a disk or a once-punctured disk on $S$. In this paper, a \textbf{curve} is  a homotopy class of essential simple closed curves on $S$.

Let $\gamma$ and $\gamma'$ be curves on $S$. The \textbf{geometric intersection number} of $\gamma$ and $\gamma'$ is defined as the minimum number of transverse intersection points among the simple representatives of $\gamma$ and $\gamma'$.

The intersection of any two curves mentioned in this paper refers to their geometric intersection number. Whenever we represent homotopy classes $\gamma$ and $\gamma'$ by simple closed curves, we assume these intersect in $i(\gamma,\gamma')$  points and we will not distinguish a homotopy class from its representatives. Two curves $\gamma$ and $\gamma'$ are \textbf{disjoint} if $i(\gamma,\gamma')=0$.   Let $A$ be a set of curves on $S$. We say that $\bigcup_{\alpha\in A}\alpha$ \textbf{fills} $S$ if the complement is a disjoint union of disks or once-punctured disks.

We call a surface which is homeomorphic to $S_{0,3}$, \textbf{a pair of pants}. 
Let $A$ be a set of pairwise disjoint curves on $S$. The \textbf{nontrivial component(s)} of the complement of the curves in $A$, denoted $(S-A)_0$, is the union of the non-pants components of the complement.

A \textbf{multicurve} $Q$ is a set of pairwise disjoint curves on $S$. Let $Q_1$ and $Q_2$ be multicurves. The intersection number of $Q_1$ and $Q_2$  is defined to be $i(Q_1,Q_2)=\sum i(\alpha_1,\alpha_2)$, where $\alpha_1$ and $\alpha_2$ are curves in $Q_1$ and $Q_2$, respectively. Observe that $Q_1\cup Q_2$ is a multicurve if and only if $i(Q_1,Q_2)=0$. 

A \textbf{pants decomposition} $P$ is a muticurve whose complement in $S$ is a disjoint union of pairs of pants. Equivalently, a pants decomposition is a maximal set of pairwise disjoint curves on $S$, that is, a maximal multicurve. A pants decomposition always contains $3g+n-3$ curves and we call this number the \textbf{complexity} $\kappa(S)$ of $S$. The \textbf{deficiency} of a multicurve $Q$ is the number $\kappa(S)-|Q|$. If $Q$ is a deficiency-$1$ multicurve then $(S-Q)_0$ is homeomorphic to either $S_{0,4}$ or $S_{1,1}$.

Let $P$ and $P'$ be pants decompositions of $S$. $P$ and $P'$ differ by an  \textbf{elementary move} if there are curves $\alpha, \alpha'$ on $S$ and a deficiency-$1$ multicurve $Q$ such that $P=\{\alpha\}\cup Q, P'=\{\alpha'\}\cup Q$ and $i(\alpha,\alpha')=2$ if $(S-Q)_0\cong S_{0,4}$ or $i(\alpha,\alpha')=1$ if $(S-Q)_0\cong S_{1,1}$; see Figure~\ref{F:elementary} for examples of elementary moves.
\begin{figure}[ht]
\begin{center}
\includegraphics[height=5cm]{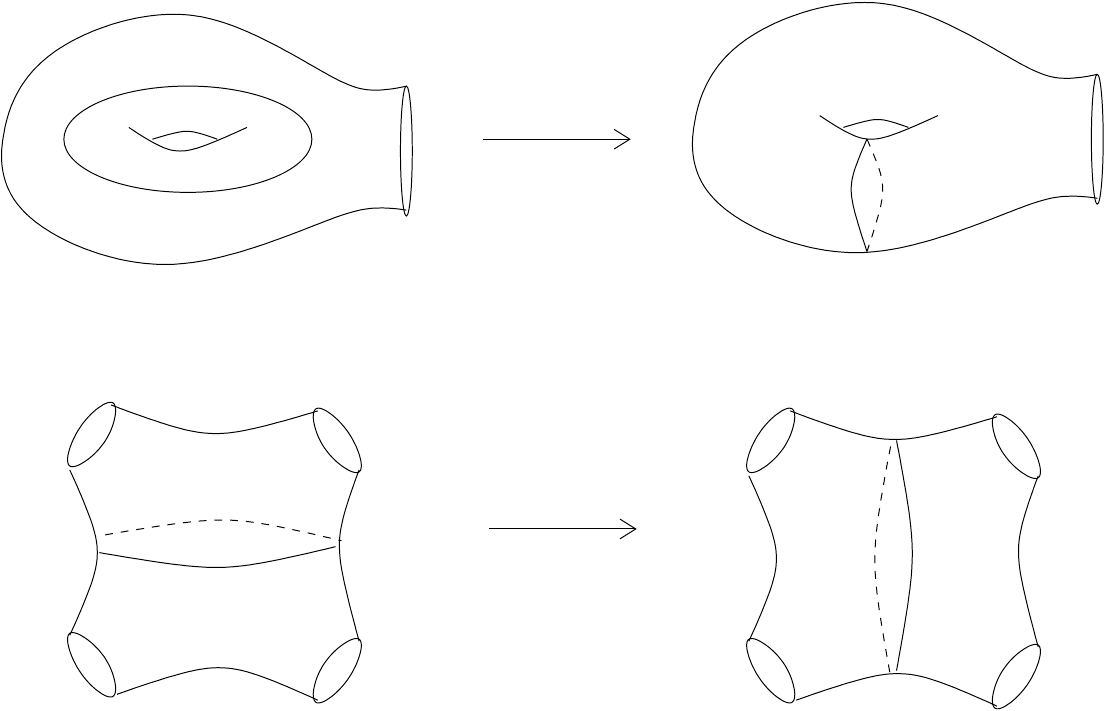} 
\caption{Two types of elementary moves.} 
\label{F:elementary}
\end{center}
\end{figure}

The \textbf{pants graph} $\mathcal{P}(S_{g,n})$  of $S_{g,n}$ is a graph with the set of vertices 
\[
V_{\mathcal{P}(S_{g,n})}=\{P\mid P~\mbox{is a pants decomposition of}~S_{g,n}\},
\] 
and the set of edges 
\[
E_{\mathcal{P}(S_{g,n})}=\{\{P,P'\}\subset V_{\mathcal{P}(S_{g,n})}\mid P, P'~\mbox{differ by an elementary move}\}.
\] 
The pants graph $\mathcal{P}(S_{g,n})$ is connected, see\cite{HT}, which we view as a geodesic metric space by requiring each edge to have length $1$. 

Given a multicurve $Q$, we let $\mathcal{P}_Q(S_{g,n})$ be the subgraph of $\mathcal{P}(S_{g,n})$ \textbf{induced} by the vertex set 
\[
V_{\mathcal{P}_Q(S_{g,n})}=\{P\in V_{\mathcal{P}(S_{g,n})}\mid Q\subset P\},
\]
that is, the largest subgraph with $V_{\mathcal{P}_Q(S_{g,n})}$ as its vertex set. When the deficiency of $Q$ is positive, it is easy to see that $\mathcal{P}_Q(S_{g,n})\cong \mathcal{P}((S_{g,n}-Q)_0)$. Let $Q_1$ and $Q_2$ be multicurves. We observe that $\mathcal{P}_{Q_1}(S_{g,n})\cap \mathcal{P}_{Q_2}(S_{g,n})\neq\emptyset$ if and only if $Q_1\cup Q_2$ is a multicurve if and only if $i(Q_1,Q_2)=0$, and in this case, $\mathcal{P}_{Q_1}(S_{g,n})\cap \mathcal{P}_{Q_2}(S_{g,n})=\mathcal{P}_{Q_1\cup Q_2}(S_{g,n})$. 

A \textbf{Farey graph} $\BF$  is a graph isomorphic to the standard Farey graph which has vertices 
\[
V=\left\{\frac{p}{q}\mid p,q\in\BZ,\frac{p}{q}~\mbox{is in lowest term}\right\}\cup\left\{\frac{1}{0}=\infty\right\},
\] 
and edges 
\[
E_{\BF}=\left\{\left\{\frac{p}{q},\frac{s}{t}\right\}\subset V\mid |pt-qs|=1\right\}.
\] 

See Figure~\ref{F:twist} for a picture of a part of the Farey graph. 

The following Lemma combines the results from (\cite{Minsky}, Section $3$), and (\cite{Mar}, Lemma 2).
\begin{lem}
\label{lem:Farey}
Let $F$ be a subgraph of $\mathcal{P}(S_{g,n})$. Then $F$ is isomorphic to a Farey graph if and only if there is a deficiency-$1$ multicurve $Q$ such that $F=\mathcal{P}_Q(S_{g,n})$.
\end{lem}
Note that as a consequence of Lemma~\ref{lem:Farey}, each edge $e$ of $\mathcal{P}(S_{g,n})$ is contained in a unique Farey graph $\mathcal{P}_Q(S_{g,n})$ where $Q=P\cap P^{\prime}$ is the deficiency-$1$ multicurve given by the intersection of its endpoints $P$ and $P^{\prime}$. 

We also see from the Lemma that $\mathcal{P}(S_{0,4}
)$ and $\mathcal{P}(S_{1,1})$ are isomorphic to a Farey graph. Let $\alpha$ and $\beta$ be two curves on $S_{0,4}$ such that $i(\alpha,\beta)=2$. Then, as pants decompositions, $\alpha$ and $\beta$ differ by an elementary move, i.e., they are two adjacent vertices in $\mathcal{P}(S_{0,4})\cong\BF$. Up to a homeomorphism on $S_{0,4}$, $\alpha$ and $\beta$ are the curves on $S_{0,4}$ shown in Figure~\ref{F:twist}. Let $T^{\frac{1}{2}}_c$ be the half-twist around a curve $c$ on $S_{0,4}$. Then we can see that 
\[
T^{\frac{1}{2}}_{\beta}(\alpha)=T^{-\frac{1}{2}}_{\alpha}(\beta),
\]
and, together with $\alpha, \beta$, these three vertices form a triangle in $\mathcal{P}(S_{0,4})$; see Figure~\ref{F:twist}. 
\begin{figure}[ht]
\begin{center}
\includegraphics[height=7cm]{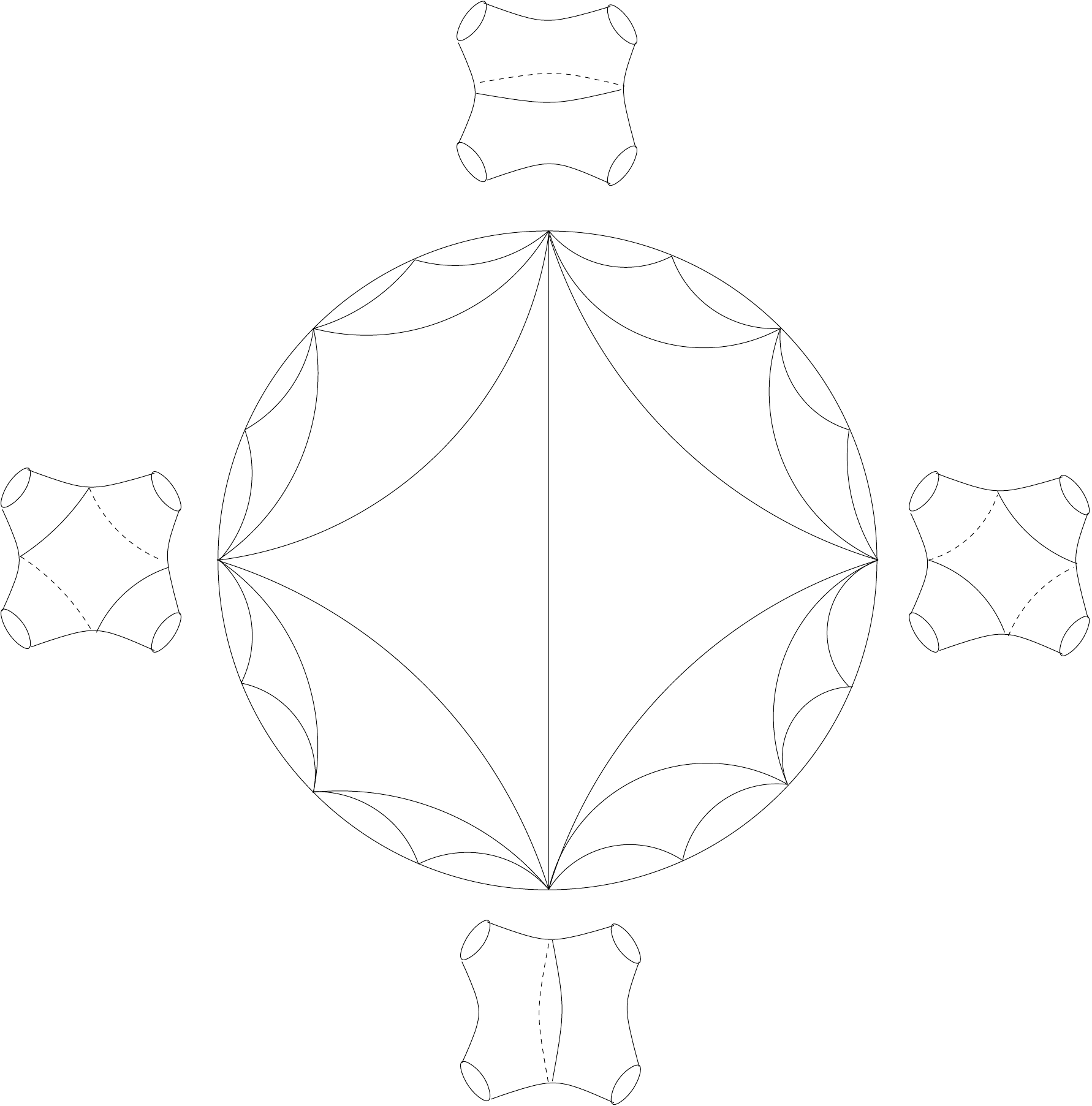} 
\caption{$\mathcal{P}(S_{0,4})$ and some curves representing its vertices.} 
\label{F:twist}
\end{center}
\begin{picture}(0,0)(0,0)

\put(215,230){\small $\alpha$}
\put(215,70){\small $\beta$}
\put(300,150){\small $T^{\frac{1}{2}}_{\beta}(\alpha)=T^{-\frac{1}{2}}_{\alpha}(\beta)$}
\put(10,150){\small $T^{\frac{1}{2}}_{\alpha}(\beta)=T^{-\frac{1}{2}}_{\beta}(\alpha)$}

\end{picture}

\end{figure}
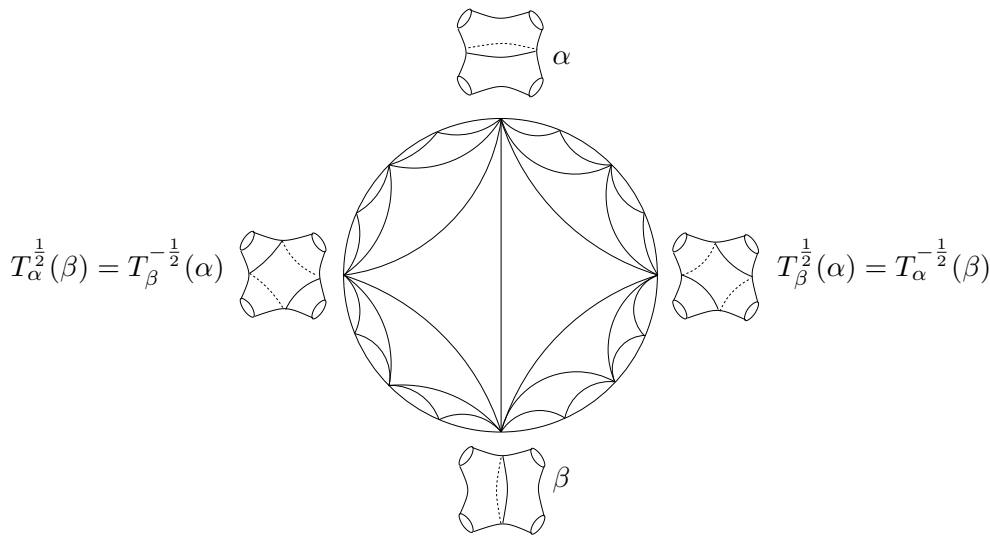
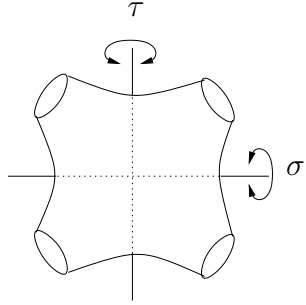
\begin{figure}[ht]
\begin{center}
\includegraphics[height=3.5cm]{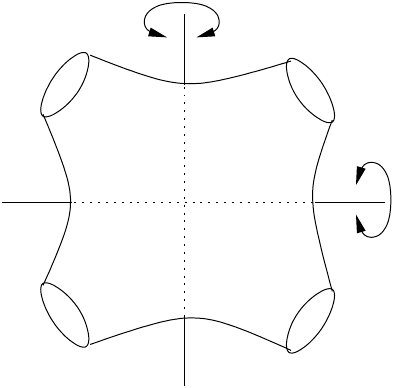} 
\caption{Two hyperelliptic involutions $\sigma$ and $\tau$ of $S_{0,4}$.} 
\label{F:hyperelliptic}
\end{center}
\begin{picture}(0,0)(0,0)

\put(190,160){$\tau$}
\put(250,100){$\sigma$}

\end{picture}

\end{figure}

Let $X_4$ be the triangle spanned by $\alpha, \beta$ and $T^{\frac{1}{2}}_{\alpha}(\beta)$. Let $\sigma$ and $\tau$ be two hyperelliptic involutions of $S_{0,4}$, see Figure~\ref{F:hyperelliptic}. Let $K=<\sigma,\tau>\cong\BZ/2\BZ\times\BZ/2\BZ<\Mod(S_{0,4})$. Note that $K$ is the subgroup of $\Mod(S_{0,4})$ containing all elements that induce identity on $\mathcal{P}(S_{0,m})$.

The case of $n=4$ in Theorem~\ref{T:main theorem} can be proved using Lemma~\ref{lem:Farey}.
\begin{prop}
\label{prop:n=4}
Given any injective simplicial map $\phi:X_4\to \mathcal{P}(S_{0,m})$. There exists a $\pi_1$-injective embedding $f:S_{0,4}\to S_{0,m}$ that induces $\phi$ and $f$ is unique up to precomposing with an element in $K$.
\end{prop}
By a \textbf{path} in $\mathcal{P}(S_{g,n})$, we always mean an edge path determined by a sequence of distinct adjacent vertices $v_1,...,v_m$ of $\mathcal{P}(S_{g,n})$.
A \textbf{cycle} in $\mathcal{P}(S_{g,n})$ is a subgraph homeomorphic to a circle. A cycle is called a \textbf{triangle}, \textbf{rectangle}, \textbf{pentagon} if it has $3, 4$ or $5$ vertices, respectively. 

A cycle $v_1,...,v_m=v_1$ is called an \textbf{alternating cycle} if any two consecutive edges are in different Farey graphs. See Figure~\ref{F:pentagon} for an example of an alternating pentagon in $\mathcal{P}(S_{0,5})$. The following is proved in \cite[Lemma 4.2]{Luo}. We note that for a simplicial map of a pentagon, however, any locally injective map is injective.
\begin{figure}[ht]
\begin{center}
\includegraphics[height=7cm]{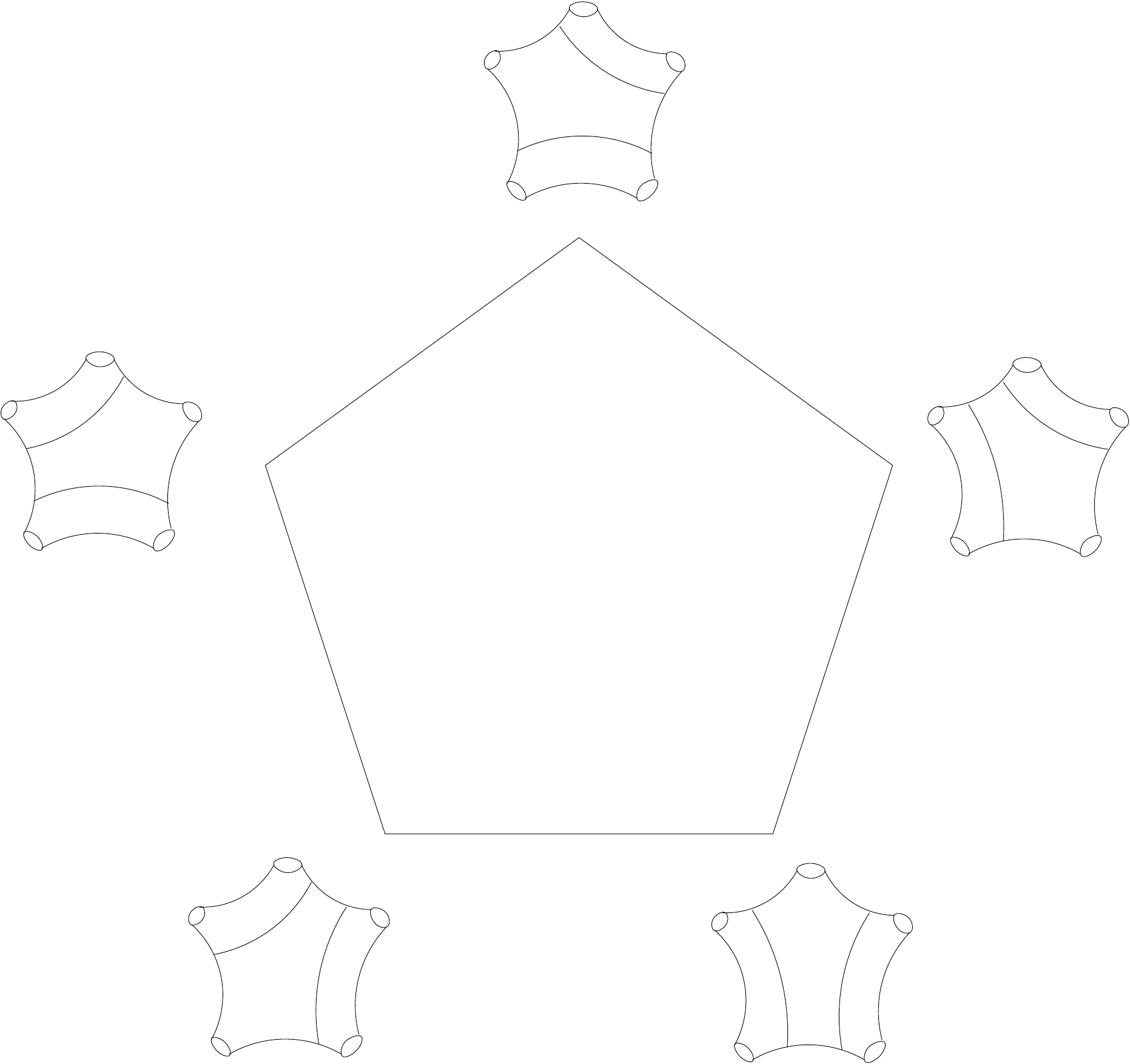} 
\caption{An alternating pentagon in $\mathcal{P}(S_{0,5})$.} 
\label{F:pentagon}
\end{center}
\end{figure}
\begin{lem}
\label{lem:pentagon}
Let $X\subset \mathcal{P}(S_{0,5})$ be the pentagon shown in Figure~\ref{F:pentagon} and let $\phi:X\to \mathcal{P}(S_{0,m})$ be a locally injective simplicial map. If $\phi(X)$ is an alternating pentagon, then there exists a deficiency-$2$ multicurve $Q$ and a homeomorphism $f:S_{0,5} \to (S_{0,m}-Q)_0$ to a component of $S_{0,m}-Q$ such that $\phi|_{X} = f^Q|_{X}$.  
\end{lem}
\section{The construction of $X_n$}
\label{sec:X_n}
In this section, we construct the finite subgraph $X_n$, for $n\geq5$.
To begin, consider $S_{0,n}$ as the double of a regular $n$-gon with vertices removed. Connect every non-adjacent pair of sides of the $n$-gon by a straight line segment and then double to obtain a curve on $S_{0,n}$. The result is $S_{0,n}$ with a set of simple closed curves $\Gamma_n$. See Figure~\ref{F:S_0_8} for the case of $S_{0,8}$ and Figure~\ref{F:S_0_5} for the case of $S_{0,5}$. 

Label the sides of the $n$-gon cyclically as $1,...,n$. In all that follows, we assume that any reference to these labels is taken modulo $n$ (thus, if $i$ is a label, so is $i+1$ and $i-1$). Given labels $i$ and $j$, write $\alpha_{i,j}$ for the curve in $\Gamma_n$ obtained from the arc connecting the $i^{th}$ side to the $j^{th}$ side. For each $i$, we call $\alpha_{i,i+2}$ \textbf{a chain curve} of $S_{0,n}$. Compare \cite[Section 3]{AL}.

Let $W\subset \Gamma_n$ be a deficiency-$m$ multicurve such that $(S_{0,n}-W)$ has only one nontrivial component $(S_{0,n}-W)_0$ homeomorphic to a sphere with $m+3$ punctures. Let $\Gamma^W_{m+3}$ be the subset of $\Gamma_n$ whose every element is disjoint from every element of $W$.
\begin{lem}
\label{lem:embed}
Let $W\subset \Gamma_n$ be a deficiency-m multicurve such that there is a nontrivial component $(S_{0,n}-W)_0$ of $(S_{0,n}-W)$ homeomorphic to a sphere with $m+3$ punctures.

Then there is a homeomorphism
\[
h: S_{0,m+3} \to (S_{0,n}-W)_0
\] 
such that $h(\Gamma_{m+3})=\Gamma^W_{m+3}$.
\end{lem}
\begin{proof}
Consider $S_{0,n}$ as the double of a regular $n$-gon as described above. Then the curves in $W$ are obtained from doubling pairwise disjoint arcs in the regular $n$-gon.  There is a component $\Delta$ of the complement of those arcs that is doubled to produce $(S_{0,n}-W)_0$.  Collapsing those arcs to points, $\Delta$ becomes an $(m+3)$-gon. Doubling, this defines a homeomorphism $h: S_{0,m+3} \to (S_{0,n}-W)_0$. The set $\Gamma^W_{m+3}$ in $(S_{0,n}-W)_0$ is then the image under $h$ of the set of curves obtained by doubling arcs connecting non-adjacent  pair of sides of the ($m+3$)-gon, i.e., $h(\Gamma_{m+3})=\Gamma^W_{m+3}$ as required. 
\end{proof}
\begin{figure}[ht]
\begin{center}
\includegraphics[height=5cm]{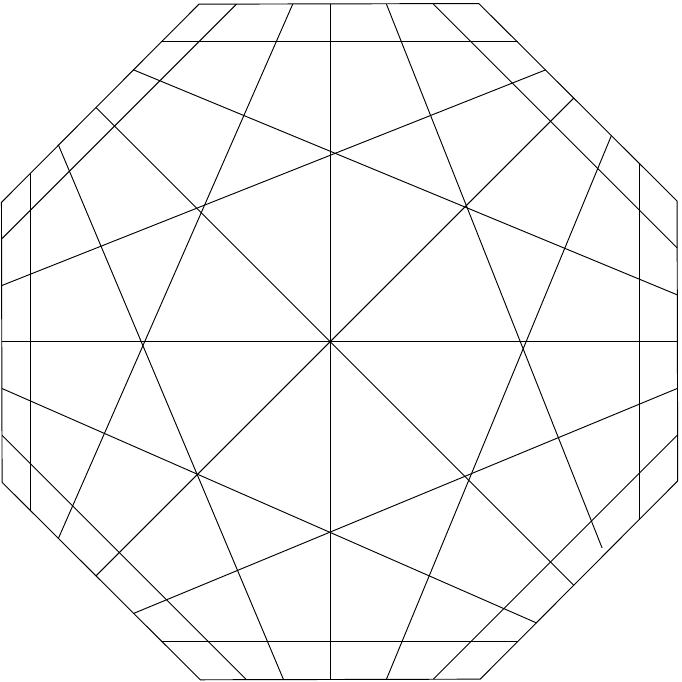} 
\caption{$S_{0,8}$ and the set of simple closed curves $\Gamma_8$.} 
\label{F:S_0_8}
\end{center}
\end{figure}
\begin{figure}[ht]
\begin{center}
\includegraphics[height=10cm]{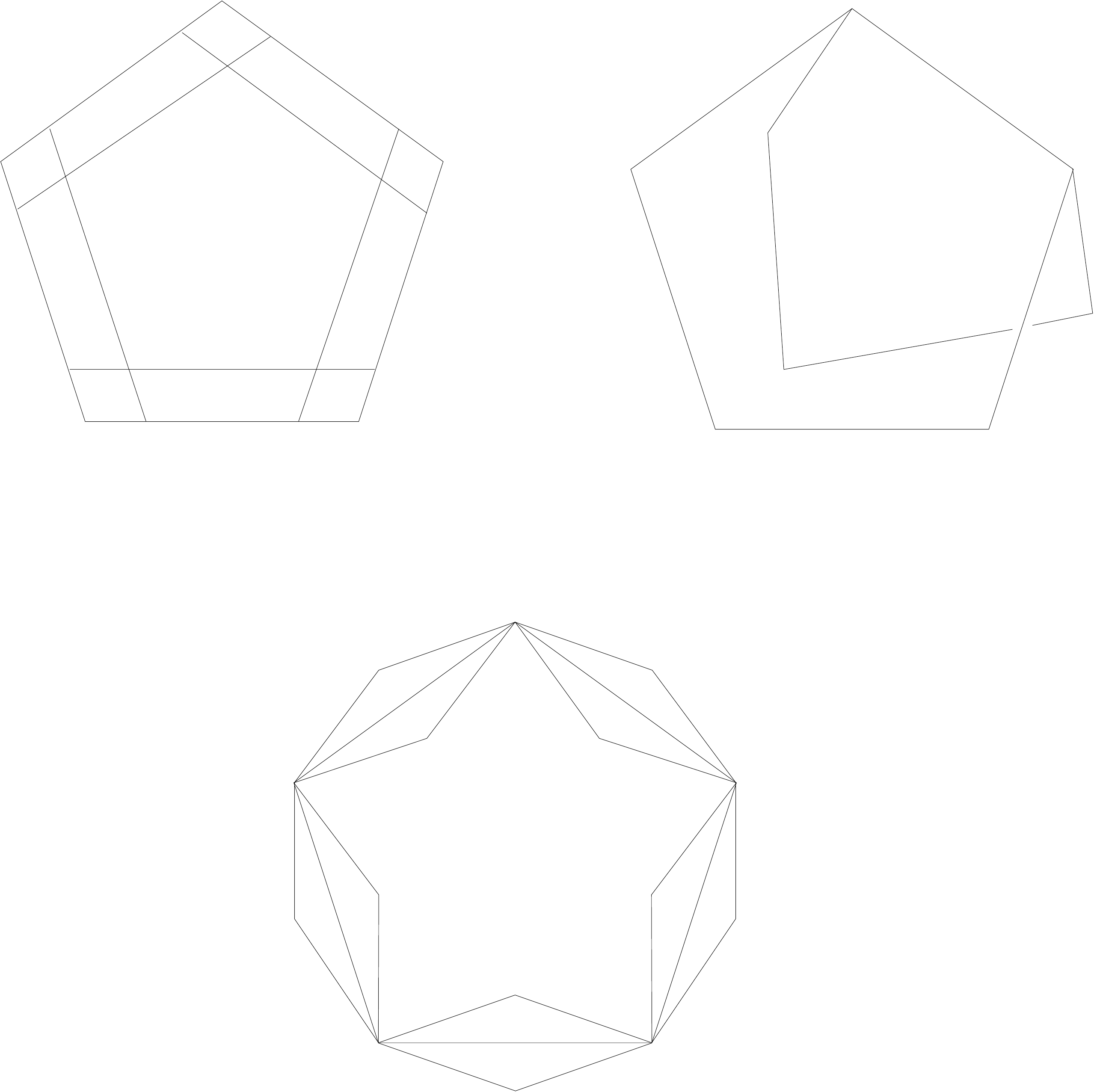} 
\caption{(Top left) $S_{0,5}$ and simple closed curves in $\Gamma_5$, (top right) $Z_5 \cup T^{\frac{1}{2}}_\beta(Z_5)$ and (bottom) the thick pentagon $\widehat{Z_5}$.} 
\label{F:S_0_5}
\end{center}

\begin{picture}(0,0)(0,0)
\put(60,330){\small $\beta$}
\put(90,350){\small $\alpha$}
\put(45,295){\small $\epsilon$}
\put(60,255){\small $\gamma$}
\put(125,230){\small $\delta$}

\put(270,355){\scriptsize $A=\left\{\alpha,\beta\right\}$}
\put(185,320){\scriptsize $E=\left\{\alpha,\gamma\right\}$}
\put(330,320){\scriptsize $B=\left\{\delta,\beta\right\}$}
\put(290,230){\scriptsize $C=\left\{\delta,\epsilon\right\}$}
\put(220,230){\scriptsize $D=\left\{\gamma,\epsilon\right\}$}

\put(330,260){\scriptsize $T^{1/2}_{\beta}(C)$}
\put(240,250){\scriptsize $T^{1/2}_{\beta}(D)$}
\put(260,320){\scriptsize $T^{1/2}_{\beta}(E)$}

\end{picture}

\end{figure}

Let $V_n$ be the set of vertices of $\mathcal{P}(S_{0,n})$ that correspond to pants decompositions consisting of curves from $\Gamma_n$.

Let $Z_n$ be a subgraph of $\mathcal{P}(S_{0,n})$ induced by $V_n$, that is, $Z_n$ is the largest subgraph with $V_n$ as its vertex set.
\begin{lem}
\label{lem:connected}
For $n\ge 5$, the subgraph $Z_n$ is finite and connected.
\end{lem}
\begin{proof}
The finiteness is obvious since the number of vertices is finite. We prove the connectedness by induction on $n$, the number of punctures. The base case when $n=5$ is true since $Z_5$ is a pentagon. 

For each $i=1,...,n$, let $v_i$ be a vertex in $Z_n$ corresponding to the pants decomposition $\left\{\alpha_{i,x}\in \Gamma_n | x\neq i-1,i,i+1\right\} $. It is not hard to see that two vertices $v_i$ and $v_{i+1}$ are connected in $Z_n$ by a path of length $n-3$. Hence any two vertices $v_i$ and $v_j$ are connected in $Z_n$. Given any vertex $v$ in $Z_n$, we will show that $v$ is connected to one of the $v_i$'s. Note that as a pants decomposition, $v$ must contain a chain curve. Choose a chain curve $\alpha=\alpha_{s,s+2}$ in $v$. The nontrivial component of $S_{0,n}-\alpha$ is homeomorphic to $S_{0,n-1}$. By Lemma~\ref{lem:embed}, $Z_n\cap P_{\alpha}(S_{0,n})\cong Z_{n-1}$. By the induction hypothesis, the vertex $v$ is connected to $ v_s=\left\{\alpha_{s,x}\in \Gamma_n\right\}$ by a path in $Z_n\cap P_{\alpha}(S_{0,n})$. Therefore we conclude that $Z_n$ is connected.
\end{proof}
Let 
\[X_5 = Z_5 \cup \bigcup_{c \in \Gamma_5} T^{\pm \frac{1}{2}}_c(Z_5),\]
where $T^{\frac{1}{2}}_c$ is a simplicial map on $\mathcal{P}(S_{0,5})$ induced by the half-twist around the chain curve $c$; see Figure~\ref{F:S_0_5}.

We see that $X_5$ is a finite subgraph consisting of $11$ alternating pentagons; one is $Z_5$ and the other $10$ are the images of $Z_5$ under the twists. Each image of $Z_5$ shares an edge with $Z_5$, thus has two edges adjacent to $Z_5$ (in particular, $X_5$ is connected). These two edges from each of the $10$ image pentagons form $10$ triangles; each triangle has one of its edges in $Z_5$. In particular, each edge of $Z_5$ has two triangles attached to it. We call $Z_5$ \textbf{the core pentagon} of $X_5$ and call $Z_5$ together with these $10$ triangles \textbf{the thick pentagon $\widehat{Z_5}$} of $X_5$. See Figure~\ref{F:S_0_5}. We also call any subgraph which is isomorphic to $\widehat{Z_5}$ a thick pentagon.
\begin{lem}
\label{lem:thick pentagon}
If $Z \subset \mathcal{P}(S_{0,m})$ is an alternating pentagon then there exists a unique thick pentagon $\widehat{Z}$ containing $Z$ and a unique subgraph $X \cong X_5$ containing $\widehat{Z}$. Moreover if $v, w\in \widehat{Z}-Z$ are connected by a length-$2$ path $\gamma$ in $X$ then $\gamma$ is the unique length-$2$ path in $\mathcal{P}(S_{0,m})$ connecting $v$ and $w$.  
\end{lem}
\begin{proof}
Each edge of $Z$ is contained in a unique Farey graph determined by the deficiency-$1$ multicurve corresponding to the intersection of its endpoints, (see \cite[Lemma 2]{Mar} or \cite[Lemma 8]{Aramayona}). Then the existence and uniqueness of $\widehat{Z}$ come from the fact that, in a Farey graph, each edge is contained in exactly two triangles. For example, in Figure~\ref{F:S_0_5}, the edge $\overline{AE}=\overline{\left\{\alpha,\beta\right\}\left\{\alpha,\gamma\right\}}$ has two triangles attaching to it and the remaining two vertices of the triangles are $T^{\frac{1}{2}}_\gamma(A)=T^{-\frac{1}{2}}_\beta(E)$ and $T^{\frac{1}{2}}_\beta(E)=T^{-\frac{1}{2}}_\gamma(A)$.

It remains to prove the existence and uniqueness of $X$. Since $Z$ is an alternating pentagon, there exists a deficiency-$2$ multicurve $Q$ such that $S_{0,m}-Q$ has exactly one nontrivial component $(S_{0,m}-Q)_0$ homeomorphic to $S_{0,5}$ and $Z\subset \mathcal{P}_Q(S_{0,m})\cong \mathcal{P}((S_{0,m}-Q)_0)\cong \mathcal{P}(S_{0,5})$, \cite[Lemma 8]{Aramayona}. Therefore we can write $Z=\left\{A=\left\{\alpha,\beta\right\}\cup Q, B=\left\{\delta,\beta\right\}\cup Q, ..., E=\left\{\alpha,\gamma\right\}\cup Q\right\}$; see Figure~\ref{F:Z}. 
\begin{figure}[ht]
\begin{center}
\includegraphics[height=4.8cm]{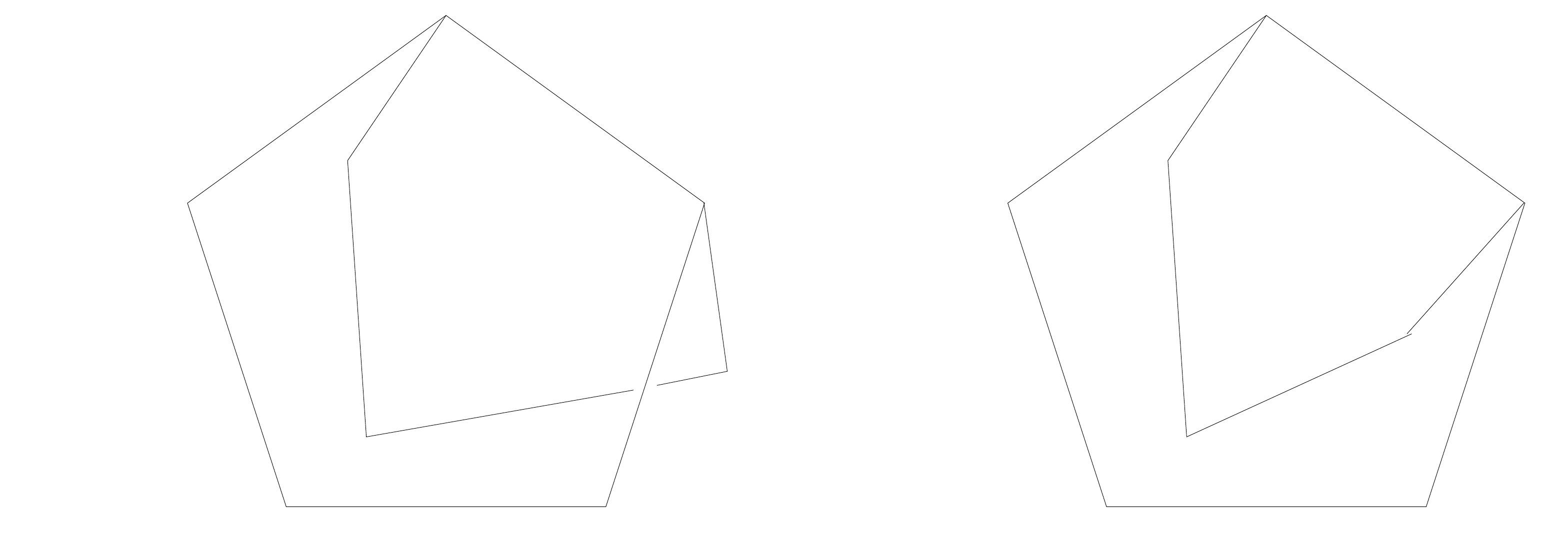} 
\caption{$Z$ and two possible pictures of $Y$.} 
\label{F:Z}
\end{center}
\begin{picture}(0,0)(0,0)

\put(110,185){\scriptsize $A=\left\{\alpha,\beta\right\}\cup Q$}
\put(3,145){\scriptsize $E=\left\{\alpha,\gamma\right\}\cup Q$}
\put(175,140){\scriptsize $B=\left\{\delta,\beta\right\}\cup Q$}
\put(150,47){\scriptsize $C=\left\{\delta,\epsilon\right\}\cup Q$}
\put(60,47){\scriptsize $D=\left\{\gamma,\epsilon\right\}\cup Q$}

\put(95,145){\scriptsize $T^{1/2}_{\beta}(E)$}
\put(190,85){\scriptsize $T^{1/2}_{\beta}(C)$}
\put(90,65){\scriptsize $F$}

\put(305,145){\scriptsize $T^{1/2}_{\beta}(E)$}
\put(330,110){\scriptsize $T^{1/2}_{\epsilon}(B)$}
\put(300,65){\scriptsize $F^{\prime}$}

\put(320,185){\scriptsize $A$}
\put(250,140){\scriptsize $E$}
\put(380,140){\scriptsize $B$}
\put(360,47){\scriptsize $C$}
\put(270,47){\scriptsize $D$}
\end{picture}

\end{figure}
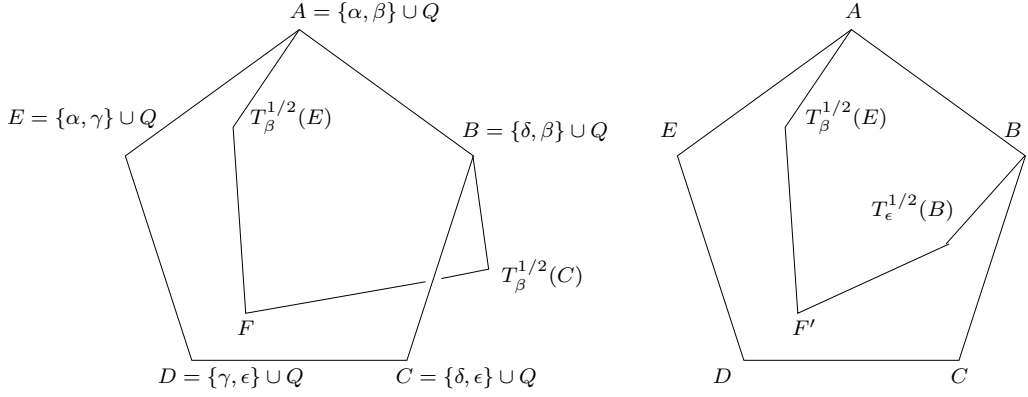

Write $\Gamma = \left\{\alpha, \beta, \delta, \epsilon, \gamma\right\}$. The subgraph 
\[X = Z \cup \bigcup_{c \in \Gamma} T^{\pm \frac{1}{2}}_c(Z),\]
contains $Z$ and $X\cong X_5$. This proves the existence of $X$.

Next, let $X^{\prime}\subset \mathcal{P}(S_{0,m})$ be a subgraph such that $Z\subset X^{\prime}$ and $X^{\prime}\cong X_5$. 
We will show that $X=X^{\prime}$. Since $\widehat{Z}$ is the unique thick pentagon containing $Z$ and $X^{\prime}\cong X_5$, it follows that $\widehat{Z}\subset X^{\prime}.
$ Let $Y\subset X^{\prime}$ be a pentagon sharing only one of its edges with $Z$. 
Considering Figure~\ref{F:Z}, we assume without loss of generality that the shared edge is  $\overline{AB}$ and $T^{\frac{1}{2}}_{\beta}(E)$ is a vertex of $Y$ adjacent to $\overline{AB}$. 
Then another edge of $Y$ adjacent to $\overline{AB}$ has either $T^{\frac{1}{2}}_{\beta}(C)$ or $T^{\frac{1}{2}}_{\epsilon}(B)$ as its endpoint.
Let $F$ be the remaining vertex of $Y$. We note that, as a pants decomposition, $F$ must contain the multicurve $Q$. 
Then by direct calculations, 
if $T^{\frac{1}{2}}_{\beta}(C)$ is the endpoint, then the vertex $F=Q\cup \left\{T^{\frac{1}{2}}_{\beta}(\gamma), T^{\frac{1}{2}}_{\beta}(\epsilon)\right\}=T^{\frac{1}{2}}_{\beta}(D)$ so $Y=T^{\frac{1}{2}}_{\beta}(Z)$, 
and if $T^{\frac{1}{2}}_{\epsilon}(B)$ is the endpoint, then there is no such $Y$ since $i(T^{\frac{1}{2}}_{\beta}(\gamma), T^{\frac{1}{2}}_{\epsilon}(\beta))=4>0$.
Hence any pentagon $Y\subset X^{\prime}$ sharing only one edge with $Z$ is $T^{\pm \frac{1}{2}}_{c}(Z)$ for some curve $c\in\Gamma$. We conclude that $X=X^{\prime}$. The last statement of the lemma follows from the proof.
\end{proof}
Let $\mathcal{G}$ be a subgraph of $\mathcal{P}(S_{0,n})$. We define \textbf{the thick graph} $\widehat{\mathcal{G}}$ for $\mathcal{G}$ to be the union of $\mathcal{G}$ with all triangles in $\mathcal{P}(S_{0,n})$ that share at least one of their edges with an edge in $\mathcal{G}$. If $e$ is an edge, we call the thick graph $\hat{e}$ for $e$ \textbf{the thick edge} which is the union of two triangles whose the common edge is $e$. 

For a vertex $v\in Z_n$, let $\st(v)$ denote the closed star of $v$ which is the union of all edges containing $v$. We define $\st_{Z_n}(v)=\st(v)\cap Z_n$, and let $\sth_{Z_n}(v)$ be the thick graph for $\st_{Z_n}(v)$.
\begin{lem}
\label{lem:edges of Z_n}
Let $v$ be a vertex of $Z_n$. Then $\st_{Z_n}(v)$ consists of $n-3$ edges. Moreover these $n-3$ edges are contained in $n-3$ distinct Farey graphs. Consequently, $\sth_{Z_n}(v)$ consists of $n-3$ thick edges and these thick edges are contained in different Farey graphs.
\end{lem}
\begin{proof}
Consider $v$ as a pants decomposition on $S_{0,n}$. Then $v$ contains $n-3$ simple closed curves. An edge which is adjacent to $v$ corresponds to an elementary move. Forgetting a closed curve $\alpha$ in $v$ gives a nontrivial component $(S_{0,n}-(v-\left\{\alpha\right\}))_0$ of $S_{0,n}$ homeomorphic to $S_{0,4}$. This nontrivial component contains two intersecting  curves in $\Gamma_n$ and one of these two curves is $\alpha$. Hence there is only one elementary move which is able to change $\alpha$ to another curve in $\Gamma_n$. Forgetting two different curves in $v$ gives two different nontrivial components homeomorphic to $S_{0,4}$. Therefore there are $n-3$ edges in $Z_n$ adjacent to $v$ and these edges are contained in $n-3$ different Farey graphs. The lemma follows.
\end{proof}
Recall the description of $S_{0,5}$ as the double of a pentagon with vertices removed. Let $e: S_{0,5} \to S_{0,5}$ be the involution exchanging the two pentagons. We observe that the involution $e$ induces a simplicial map on $\mathcal{P}(S_{0,5})$ whose the restriction to $Z_5$ is the identity and which restricts to a symmetry of $X_5$.
\begin{lem}
\label{lem:sym}
Let $G=\Sym(X_5,Z_5)$ be the subgroup of the symmetry group $\Sym(X_5)$ of $X_5$ consisting of all elements that fix $Z_5$ pointwise. Then $G\cong \BZ/2\BZ$ generated by $e$.
\end{lem}
\begin{proof}
Recall the thick pentagon $\widehat{Z_5}$ in Figure~\ref{F:S_0_5}. There are $10$ vertices of the triangles outside $Z_5$ and we number them as in Figure~\ref{F:sym}. 
\begin{figure}[ht]
\begin{center}
\includegraphics[height=5cm]{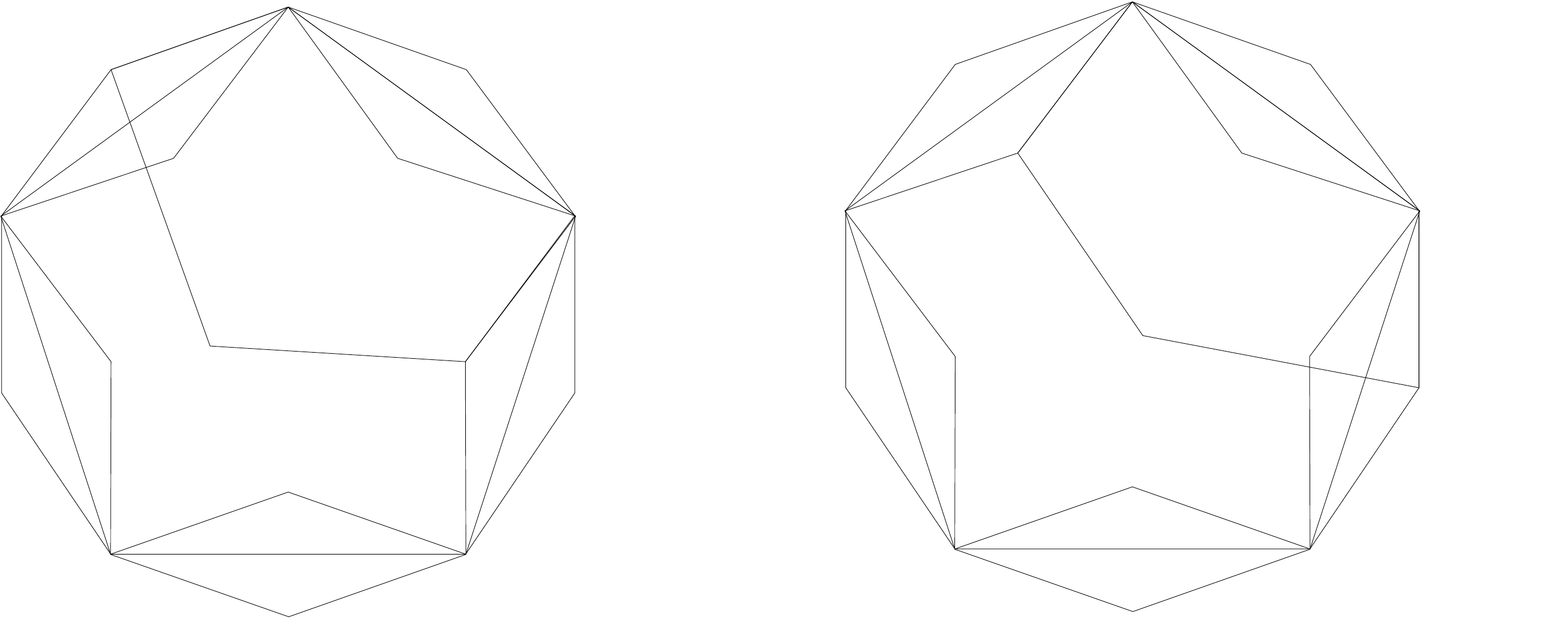} 
\caption{Thick pentagon $\widehat{Z_5}$ and the $10$ labeled vertices.} 
\label{F:sym}
\end{center}
\begin{picture}(0,0)(0,0)

\put(35,180){\small $1$}
\put(125,180){\small $3$}
\put(55,145){\small $2$}
\put(95,145){\small $4$}
\put(40,115){\small $10$}
\put(115,115){\small $6$}
\put(80,80){\small $8$}
\put(80,40){\small $7$}
\put(10,85){\small $9$}
\put(150,95){\small $5$}

\put(230,180){\small $1$}
\put(320,180){\small $3$}
\put(257,145){\small $2$}
\put(295,145){\small $4$}
\put(235,115){\small $10$}
\put(310,115){\small $6$}
\put(275,80){\small $8$}
\put(275,40){\small $7$}
\put(205,85){\small $9$}
\put(345,95){\small $5$}
\end{picture}

\end{figure}
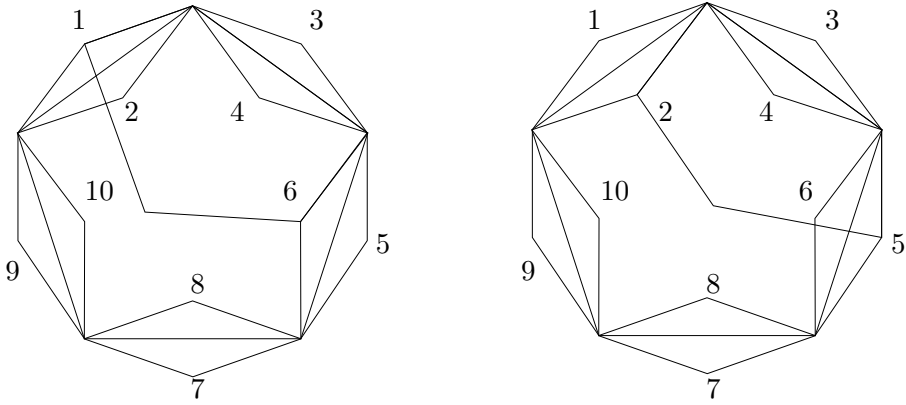
Let $V$ be the set of these $10$ vertices.
Each vertex in $V$ is paired by two pentagons in $X_5$ to two vertices in $V$, for example, $\left\{1,6\right\}$ is a pair because $1$ and $6$ are in the same image pentagon shown in Figure~\ref{F:sym}. This forms $10$ pairs of vertices, namely, $\left\{1,6\right\}, \left\{2,5\right\}, \left\{1,8\right\}, \left\{2,7\right\}, \left\{3,8\right\}, \left\{4,7\right\}, \left\{3,10\right\}, \left\{4,9\right\}, \left\{5,10\right\}, \left\{6,9\right\}$. 

Let $\BS_{\left\{x,y\right\}}$ be the symmetric group on two letters $x$ and $y$. Note that $\BS_{\left\{x,y\right\}}\cong \BZ/2\BZ$ generated by $\sigma_{\left\{x,y\right\}}$ which interchanges $x$ and $y$.

Let $H=\BS_{\left\{1,2\right\}}\times \BS_{\left\{3,4\right\}}\times \BS_{\left\{5,6\right\}}\times \BS_{\left\{7,8\right\}}\times \BS_{\left\{9,10\right\}}\cong (\BZ/2\BZ)^5$.
There is a natural injective homomorphism $\eta : G \to H$. Let $g$ be a nonidentity element in $G$. Without loss of generality, we assume that $g(1)=2$. By Lemma~\ref{lem:thick pentagon}, $g$ maps the pentagon in $X_5$ containing $1$ and $6$ to the pentagon in $X_5$ containing $2$ and $5$; see Figure~\ref{F:sym}. Thus $g(5)=6$. By similar argument and the above observation of how vertices in $V$ are paired to each others, we see that $\eta(g)=(\sigma_{\left\{1,2\right\}}, \sigma_{\left\{3,4\right\}}, \sigma_{\left\{5,6\right\}}, \sigma_{\left\{7,8\right\}}, \sigma_{\left\{9,10\right\}})$. Hence $\eta(G)\cong \BZ/2\BZ$ and so $G\cong \BZ/2\BZ$  generated by $e$.
\end{proof}
For $n>5$, we construct $X_n$ as follows.
Let $W\subset \Gamma_n$ be a deficiency-$2$ multicurve such that the nontrivial component $(S_{0,n}-W)_0$ of $S_{0,n}-W$ is homeomorphic to $S_{0,5}$. By Lemma~\ref{lem:embed}, we have a homeomorphism of pairs $h:(S_{0,5},\Gamma_5) \to ((S_{0,n}-W)_0,\Gamma^W_5)$.
Let 
\[ X_5^W = h^W(X_5) = \{ h(u) \cup W \mid u \in X_5 \},\]
where $h^W:P(S_{0,5})\to P(S_{0,n})$ is the induced map of $h$ as defined in Section~\ref{sec:Introduction} by $h^W(u)=h(u)\cup W$.

Finally we let 
\[
X_n = Z_n \cup \bigcup_W X^W_5,
\]
where the union is taken over all deficiency-$2$ multicurves in $\Gamma_n$ with a $5$-punctured sphere component. 
\begin{lem}
\label{lem:X_n}
For $n\geq6, X_n$ has the following properties.
\begin{enumerate}
\item $X_n \subset \mathcal{P}(S_{0,n})$ is connected.
\item For each chain curve $\alpha_i$, $i=1,...,n$, let $X^i_{n-1}=X_n \cap \mathcal{P}_{\alpha_i}(S_{0,n})$. Then $X^i_{n-1}\cong X_{n-1}$. Moreover, this isomorphism is induced by $h: S_{0,n-1}\to(S_{0,n}-\alpha_i)_0$ as $h^{\alpha_i}(v) = h(v) \cup \{\alpha_i\} \in X_{n-1}^i$.
\item If $n\geq 7$ and $\alpha_i, \alpha_j$ are disjoint chain curves then $X^{ij}_{n-2}=X_{n}\cap \mathcal{P}_{\left\{\alpha_i,\alpha_j\right\}}(S_{0,n})\cong X_{n-2}$ with isomorphism $h^{\{\alpha_i,\alpha_j \}}$.
\end{enumerate}
\end{lem}
\begin{proof}
1. Since the core pentagon of each $X^W_5$ is in $Z_n$, $X^W_5$ is connected to $Z_n$ in $X_n$. By Lemma~\ref{lem:connected}, $Z_n$ is connected. Hence $X_n$ is connected.

2. Fix a chain curve $\alpha_i$. By Lemma~\ref{lem:embed}, there is a homeomorphism
$
h: S_{0,n-1} \to (S_{0,n}-\alpha_i)_0
$
such that $h(\Gamma_{n-1})=\Gamma^{\alpha_i}_{n-1}$
and $h$ induces an isomorphism from $Z_{n-1}$ to $Z_n\cap \mathcal{P}_{\alpha_i}(S_{0,n})$. 
Moreover, for every deficiency-$2$ multicurve $W_0\subset\Gamma_{n-1}$, $W=h(W_0)\cup\{\alpha_i\}$ is a deficiency-$2$ multicurve in $\Gamma_n$ and $h$ induces an isomorphism from $X^{W_0}_5$ to $X^W_5$. 
Therefore $h$ induces a simplicial injection from $X_{n-1}$ to $X_n\cap \mathcal{P}_{\alpha_i}(S_{0,n})$.
We need to show that this is surjective.

Recall the definition of $X_n$ above.
Let $W$ be a deficiency-$2$ multicurve in $\Gamma_n$ with $(S_{0,n}-W)_0\cong S_{0,5}$.
If $i(\alpha_i,W)\neq 0$, then $X^W_5\cap \mathcal{P}_{\alpha_i}(S_{0,n})=\emptyset$. 
For the rest of the proof we assume that $i(\alpha_i,W)= 0$.
To prove the surjectivity, we will show that any vertex $u\in X^W_5\cap\mathcal{P}_{\alpha_i}(S_{0,n})$ is in $h^{\alpha_i}(X^{W_0}_5)$ for some deficiency-$2$ multicurve $W_0$ in $\Gamma_{n-1}$. We consider the following two cases.

If $\alpha_i \in W$, then $W_0' = W - \{\alpha_i\}$ is a deficiency-$2$ multicurve in $(S_{0,n}-\alpha_i)_0$. Setting $W_0 = h^{-1}(W_0')$, we have $W = h(W_0) \cup \{\alpha_i\}$.  Thus $X_5^W$ is the image of $X_5^{W_0}$ by $h^{\alpha_i}$.

Suppose $\alpha_i\notin W$. 
Then $\alpha_i$ is contained in $(S_{0,n}-W)_0$, moreover $\alpha_i\in \Gamma^W_5$.
Hence, by the definition of $X^W_5$,  $X^W_5\cap \mathcal{P}_{\alpha_i}(S_{0,n})\neq\emptyset$.
Let $u$ be a vertex in $X^W_5\cap \mathcal{P}_{\alpha_i}(S_{0,n})$.
Consider $u$ as a pants decomposition.
We claim that there is a deficiency-$2$ multicurve $W^{\prime\prime}$ such that $\alpha_i\in W^{\prime\prime}$, $u\in X^{W^{\prime\prime}}_5\cap \mathcal{P}_{\alpha_i}(S_{0,n})$, and so from the case that $\alpha_i \in W^{\prime\prime}$, we see that $u$ is in the image of $h^{\alpha_i}$, as required.
We prove the claim as follows.
Since $u\in X^W_5\cap \mathcal{P}_{\alpha_i}(S_{0,n})$ and $\alpha_i\notin W$, $u=W\cup \left\{\alpha_i\right\}\cup \left\{x\right\}$ for some simple closed curve $x=y$ or $T^{\pm\frac{1}{2}}_\beta(y)$, with $y, \beta\in \Gamma^W_n$. 
Since $S_{0,n}$ has complexity at least $3$ and $\alpha_i$ is a chain curve, there exists a close curve $\gamma\in W$ such that $W^{\prime\prime}=(W-\left\{\gamma\right\})\cup\left\{\alpha_i\right\}$ is a deficiency-$2$ multicurve with nontrivial component $(S_{0,n}-W^{\prime\prime})_0\cong S_{0,5}$.
Then $u=W^{\prime\prime} \cup\left\{\gamma\right\}\cup \left\{x\right\}\in X^{W^{\prime\prime}}_5\cap \mathcal{P}_{\alpha_i}(S_{0,n})$ as desired.

3. The statement is proved by applying 2 twice.
\end{proof}
\section{The proof for $S_{0,5}$}
\label{sec:base case}
We prove the main theorem for $n=5$.
\begin{lem}
\label{lem:base}
Let $X_5\subset \mathcal{P}(S_{0,5})$ be as above. 
Then for any punctured sphere $S_{0,m}$ and any  injective simplicial map 
\[
\phi:X_5 \to \mathcal{P}(S_{0,m}),
\]
there exists a $\pi_1$-injective embedding $f:S_{0,5} \to S_{0,m}$, unique up to isotopy, that induces $\phi$.
\end{lem}
\begin{proof}
We show that $\phi$ maps the core pentagon $Z_5$ of the thick pentagon $\widehat{Z_5}$  in $X_5$ to an alternating pentagon in $\mathcal{P}(S_{0,m})$.

Let $F$ be a Farey graph in $\mathcal{P}(S_{0,m})$. 
We claim that $\phi$ cannot map any three consecutive edges of $Z_5$ into $F$.
To find a contradiction, we assume that $\phi$ maps three consecutive edges $\overline{ABCD}$ of $Z_5$ into $F$. 
Since $\phi$ is locally injective, $\phi$ maps all six triangles attaching to these edges to distinct triangles in $F$. Up to an automorphism of $F$, the image of the six triangles must be one of the two pictures in Figure~\ref{F:1}.
This implies that the distance of $\phi(A)$ and $\phi(D)$ in $F$ is greater than the distance of $\phi(A)$ and $\phi(D)$ in $\mathcal{P}(S_{0,m})$, that is, 
\[
\dist_{F}(\phi(A),\phi(D))=3>2=\dist_{\phi(X_5)}(\phi(A),\phi(D))\geq\dist_{\mathcal{P}(S_{0,m})}(\phi(A),\phi(D)).
\]
This is a contradiction to the fact, proven in \cite{APS}, that $F$ is isometrically embedded.
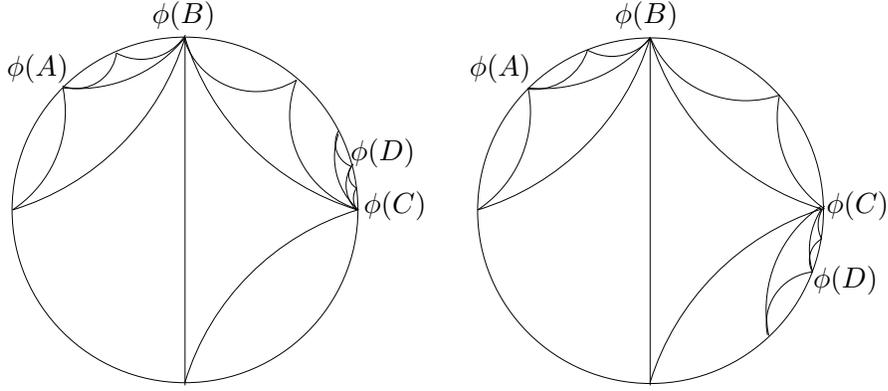
\begin{figure}[ht]
\begin{center}
\includegraphics[height=5cm]{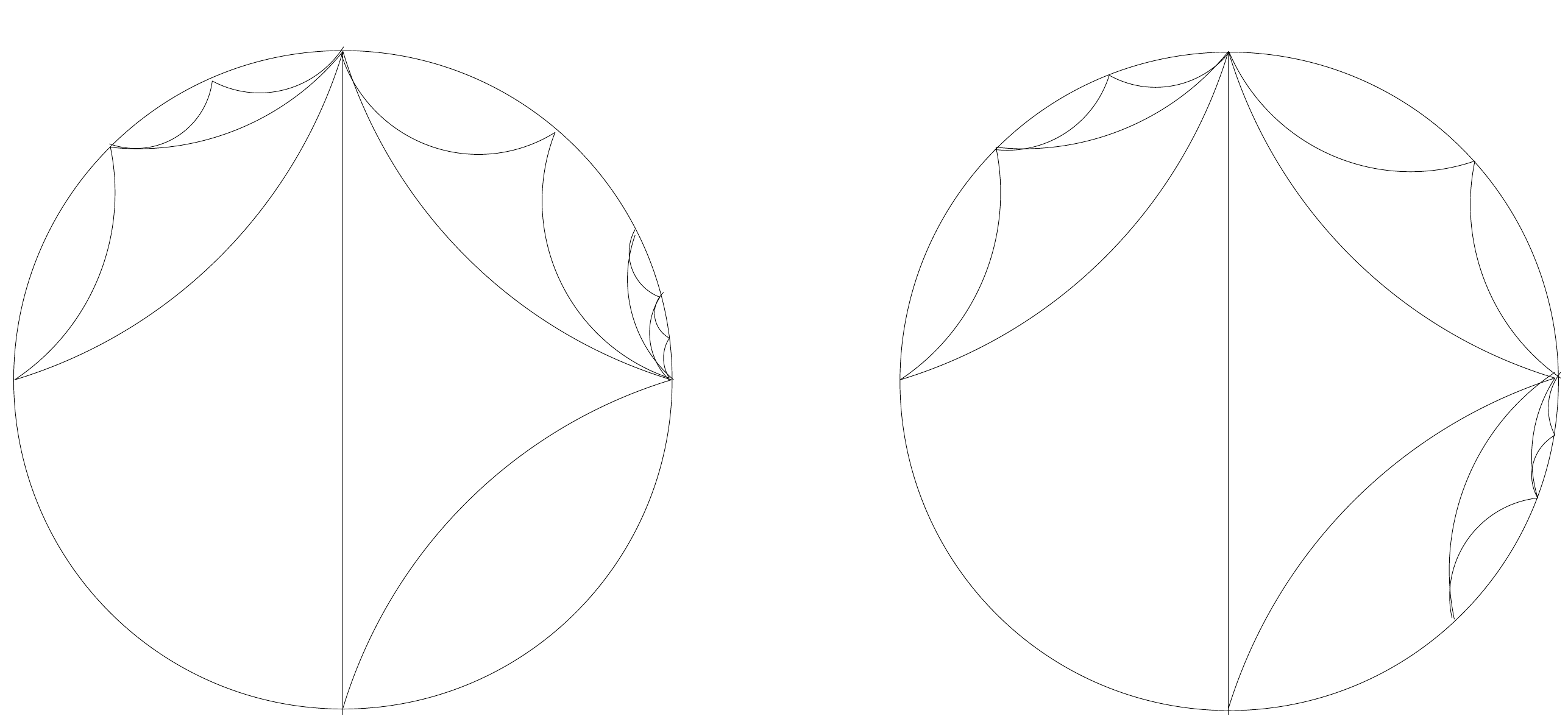} 
\caption{Possible images in $F$ of three consecutive edges of $\hat{Z_5}$ and triangles of $\hat{Z_5}$ adjacent to them.} 
\label{F:1}
\end{center}
\begin{picture}(0,0)(0,0)

\put(40,180){\small $\phi(A)$}
\put(95,200){\small $\phi(B)$}
\put(175,128){\small $\phi(C)$}
\put(170,148){\small $\phi(D)$}

\put(215,180){\small $\phi(A)$}
\put(270,200){\small $\phi(B)$}
\put(350,128){\small $\phi(C)$}
\put(345,100){\small $\phi(D)$}

\end{picture}

\end{figure}

Next we claim that $\phi$ cannot map any two adjacent edges of $Z_5$ into $F$. To find a contradiction, we assume that $\phi$ maps $\overline{ABC}$ into $F$. The locally injectivity of $\phi$ implies that the two edges $\overline{\phi(A)\phi(B)}$ and $\overline{\phi(B)\phi(C)}$ are separated by at least three triangles in $F$. Let $W$ be the deficiency-$1$ multicurve defining $F$, so that  $\phi(A)=W\cup\left\{a\right\}$, $\phi(B)=W\cup\left\{b\right\}$. Then $\phi(C)=W\cup\left\{T^{k}_{b}(a)\right\}, k\geq 3/2$; see Figure~\ref{F:2}. By the previous claim the edges $\overline{\phi(A)\phi(E)}$ and $\overline{\phi(C)\phi(D)}$ are not in $F$. Then $\phi(E)=W^{\prime}\cup\left\{a\right\}$, $\phi(D)=W^{\prime\prime}\cup\left\{T^{k}_{b}(a)\right\}$; see Figure~\ref{F:2}. Since $i(a,T^{k}_{b}(a))>2$ when $k\geq3/2$, $i(\phi(E),\phi(D))> 2$. Hence there is no elementary move from $\phi(D)$ to $\phi(E)$; i.e. $\phi(D)$ and $\phi(E)$ are not connected by an edge.  This is a contradiction, and so we conclude that $\phi(Z_5)$ is an alternating pentagon. Moreover $\phi(\widehat{Z_5})$ is isomorphic to $\widehat{Z_5}$, i.e., $\phi$ is injective on $\widehat{Z_5}$. For any two vertices $v$ and $w$ in $\widehat{Z_5}-Z_5$ that are connected by a length-$2$ path in $X_5$, $\phi(v)$ and $\phi(w)$ are also connected by a length-$2$ path in $\phi(X_5)$. Appealing to the last statement of Lemma~\ref{lem:thick pentagon}, $\phi(X_5)$ is isomorphic to $X_5$, i.e., $\phi$ is injective on $X_5$.
\begin{figure}[ht]
\begin{center}
\includegraphics[height=5cm]{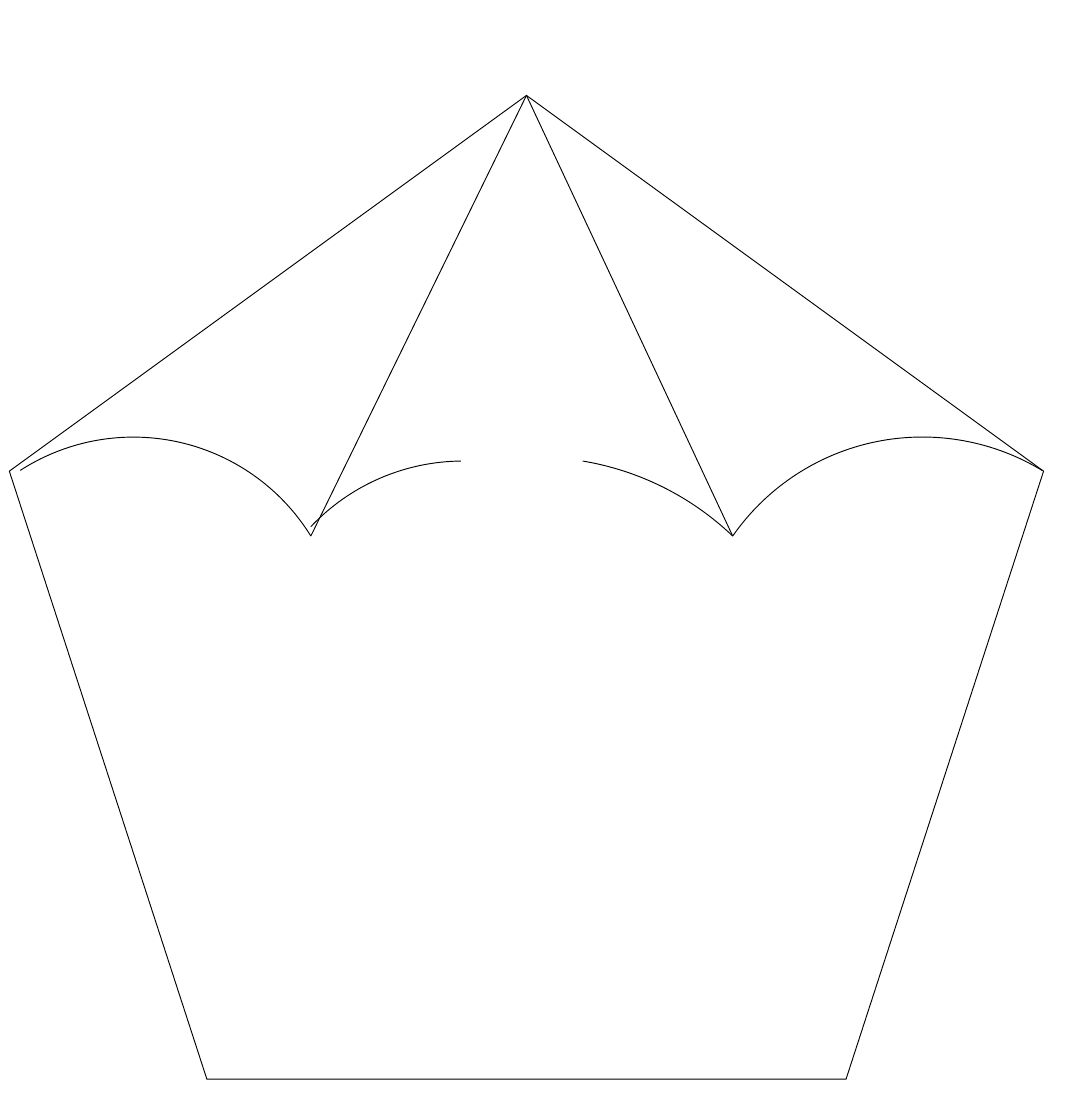} 
\caption{Image of two adjacent edges of $Z_5$ in $F$.} 
\label{F:2}
\end{center}
\begin{picture}(0,0)(0,0)

\put(170,180){\footnotesize $\phi(B)=W\cup\left\{b\right\}$}
\put(55,130){\footnotesize $\phi(A)=W\cup\left\{a\right\}$}
\put(260,130){\footnotesize $\phi(C)=W\cup\left\{T^k_b(a)\right\}$}
\put(75,50){\footnotesize $\phi(E)=W^{\prime}\cup\left\{a\right\}$}
\put(240,50){\footnotesize $\phi(D)=W^{\prime\prime}\cup\left\{T^k_b(a)\right\}$}
\put(145,110){\footnotesize $W\cup\left\{T^{\frac{1}{2}}_b(a)\right\}$}
\end{picture}

\end{figure}
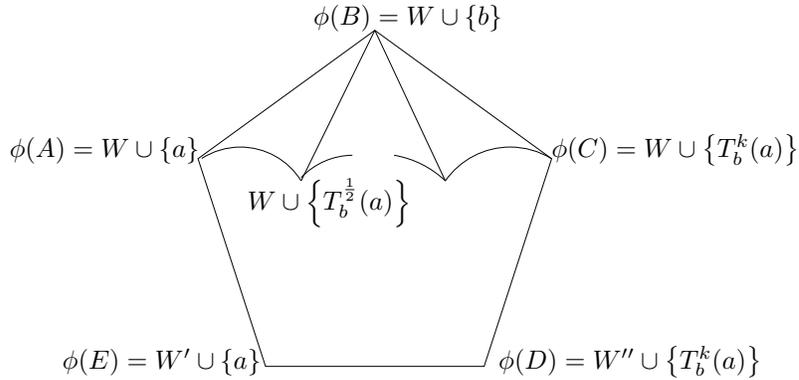
By Lemma~\ref{lem:pentagon}, there exists a deficiency-$2$ multicurve $Q$ and a homeomorphism $f:S_{0,5} \to (S_{0,m}-Q)_0$ to the nontrivial component of $S_{0,m}-Q$ such that $\phi|_{Z_5} = f^Q|_{Z_5}$.  Moreover, $f$ is unique up to precomposing by $e$ (and isotopy) since the pointwise stabilizer of $Z_5$ in $\Mod(S_{0,5})$ is generated by $e$.
So $f^Q(X_5)$ is a subgraph of $\mathcal{P}(S_{0,m})$ containing $f^Q(Z_5)=\phi(Z_5)$. By Lemma~\ref{lem:thick pentagon}, $f^Q(X_5)=\phi(X_5)$. Since ${(f^Q)}^{-1}\circ \phi$ is the identity map on $Z_5$, Lemma~\ref{lem:sym} implies that either ${(f^Q)}^{-1}\circ \phi=\id_{X_5}$ or $e\circ {(f^Q)}^{-1}\circ \phi=\id_{X_5}$. Hence $(f^Q\circ e)_{|{X_5}}={(f\circ e)}^Q_{|{X_5}}=\phi$ or ${f^Q}_{|{X_5}}=\phi$. In either case $\phi$ is induced by an embedding.
\end{proof}
\section{The general case}
\label{sec:general case}
We now proceed to the proof of the main theorem in the general case.
\begin{proof}[Proof of Theorem~\ref{T:main theorem}]
We prove the main theorem by induction on $n$.

Lemma~\ref{lem:base} proves the base case when $n=5$.

Suppose the theorem is true for an $n\geq 5$. Consider an injective simplicial map $\phi: X_{n+1} \to \mathcal{P}(S_{0,m})$. For each chain curve $\alpha_i$ in $S_{0,n+1}, i=1,...,n+1$, recall $X^i_n=X_{n+1}\cap \mathcal{P}_{\alpha_i}(S_{0,n+1})$. By Lemma~\ref{lem:X_n}, $X^i_n\cong X_n$ and $X^i_n\subset \mathcal{P}_{\alpha_i}(S_{0,n+1})\cong \mathcal{P}((S_{0,n+1}-\alpha_i)_0)\cong \mathcal{P}(S_{0,n})$. Given a vertex $u\in X^i_n$, consider $u$ as a pants decomposition of $S_{0,n+1}$ so that $u_{\alpha_i}=u-\left\{\alpha_i\right\}$ is a pants decomposition of $(S_{0,n+1}-\alpha_i)_0$ . Define an injective simplicial map
\[
\phi_i: X^i_n \to \mathcal{P}(S_{0,m})
\]
by $\phi_i(u)=\phi(u)$.
By the induction hypothesis, there is an embedding $f_i:(S_{0,n+1}-\alpha_i)_0 \to S_{0,m}$, unique up to isotopy, such that $f_i$ induces $\phi_i$, i.e., $f_i$ has the following properties:
\begin{enumerate}
\item there is a deficiency-$(n-3)$ multicurve $Q_i$ in $S_{0,m}$ such that $S_{0,m}-Q_i$ has only one nontrivial component $(S_{0,m}-Q_i)_0\cong S_{0,n}$ and $f_i((S_{0,n+1}-\alpha_i)_0)=(S_{0,m}-Q_i)_0$,
\item the simplicial map $f^{Q_i}_i:\mathcal{P}_{\alpha_i}(S_{0,n+1})\to \mathcal{P}(S_{0,m})$ defined by 
\[
f^{Q_i}_i(u)=f_i(u_{\alpha_i})\cup Q_i
\]
satisfies $f_i^{Q_i}|_{X_n^i}=\phi_i$.
\end{enumerate}
Under all the hypotheses above, we prove the following three lemmas.
\begin{lem}
\label{lem:1}
 If $\alpha_i$ and $\alpha_j$ are disjoint chain curves then $i(Q_i,Q_j)=0$.
\end{lem}
\begin{proof}
Since $n+1\geq 6$, $(S_{0,n+1}-\left\{\alpha_i,\alpha_j\right\})_0\cong S_{0,n-1}$ with $n-1\geq 4$. Since $\left\{\alpha_i,\alpha_j\right\}\subset \Gamma_{n+1}$, $Z_{n+1}\cap \mathcal{P}_{\left\{\alpha_i,\alpha_j\right\}}(S_{0,n+1})\neq\emptyset$ and contains an edge $e$. Since $e$ is an edge in 
\[
\mathcal{P}_{\left\{\alpha_i,\alpha_j\right\}}(S_{0,n+1})=\mathcal{P}_{\left\{\alpha_i\right\}}(S_{0,n+1})\cap \mathcal{P}_{\left\{\alpha_j\right\}}(S_{0,n+1}),
\]
$\phi(e)=\phi_i(e)\subset \mathcal{P}_{Q_i}(S_{0,m})$ and $\phi(e)=\phi_j(e)\subset \mathcal{P}_{Q_j}(S_{0,m})$. Thus 
\[
\mathcal{P}_{Q_i}(S_{0,m})\cap \mathcal{P}_{Q_j}(S_{0,m})\neq\emptyset.
\]
Hence $i(Q_i,Q_j)=0$.
\end{proof}
\begin{lem}
\label{lem:2}
If $\alpha_i,\alpha_j,\alpha_k$ are pairwise disjoint chain curves and $u\in Z_{n+1}\cap \mathcal{P}_{\left\{\alpha_i,\alpha_j,\alpha_k\right\}}(S_{0,n+1})$ is any vertex, then the $n-2$ thick edges in $\sth_{Z_{n+1}}(u)$ from Lemma~\ref{lem:edges of Z_n} are mapped into $n-2$ distinct Farey graphs by $\phi$.
\end{lem}
\begin{proof}
Lemma~\ref{lem:edges of Z_n} shows that the $n-2$ thick edges in $\sth_{Z_{n+1}}(u)$ are contained in distinct Farey graphs. 
Then $\sth_{X_{n+1}}(u) \cap \mathcal{P}_{\alpha_i}(S_{0,n+1})$ contains $n-3$ thick edges and $f^{Q_i}_i$ maps the $n-3$ distinct Farey graphs containing these thick edges to distinct Farey graphs in $\mathcal{P}_{Q_i}(S_{0,m})$. The same is true if $i$ is replaced by $j$ or $k$.
%
%

Now for any two thick edges $\hat{e}_1$ and $\hat{e}_2$ in $\sth_{Z_{n+1}}(u)$, $\hat{e}_1$ and $\hat{e}_2$ are both contained in at least one of $\mathcal{P}_{\alpha_i}(S_{0,n+1})$, $\mathcal{P}_{\alpha_j}(S_{0,n+1})$ or $\mathcal{P}_{\alpha_k}(S_{0,n+1})$. 
Assume that $\hat{e}_1$ and $\hat{e}_2$ are contained in $\mathcal{P}_{\alpha_i}(S_{0,n+1})$, then $f^{Q_i}_i(\hat{e}_1)=\phi(\hat{e}_1)$ and $f^{Q_i}_i(\hat{e}_2)=\phi(\hat{e}_2)$ are in different Farey graphs.
\end{proof}
\begin{lem}
\label{lem:3}
If $\alpha_i,\alpha_j,\alpha_k$ are pairwise disjoint chain curves then $Q_i\neq Q_j\neq Q_k\neq Q_i$.
\end{lem}
\begin{proof}
%

Suppose $Q_i=Q_j$. 
Let $u$ be a vertex in $Z_{n+1}\cap \mathcal{P}_{\left\{\alpha_i,\alpha_j,\alpha_k\right\}}(S_{0,n+1})$. %

As in the previous proof, $f_i^{Q_i}$ and $f_j^{Q_j}$ map the $n-3$ thick edges of $\sth_{X_{n+1}}(u) \cap \mathcal{P}_{\alpha_i}(S_{0,n+1})$ and $\sth_{X_{n+1}}(u) \cap \mathcal{P}_{\alpha_j}(S_{0,n+1})$ into $n-3$ thick edges in $\mathcal{P}_{Q_i}(S_{0,m})$ and $\mathcal{P}_{Q_j}(S_{0,m})=\mathcal{P}_{Q_i}(S_{0,m})$, respectively.
There are $n-4$ thick edges in $\sth_{X_{n+1}}(u) \cap \mathcal{P}_{\alpha_i}(S_{0,n+1}) \cap \mathcal{P}_{\alpha_j}(S_{0,n+1})$ and $f^{Q_i}_i$ agrees with $f^{Q_j}_j$ on these thick edges because $f^{Q_i}_i(v)=\phi(v)=f^{Q_j}_j(v)$, for any vertex $v\in X^i_n\cap X^j_n$. 
Let $\widehat{e_i}$ be the thick edge in $\sth_{X_{n+1}}(u) \cap \mathcal{P}_{\alpha_i}(S_{0,n+1})$ but not in $\sth_{X_{n+1}}(u)\cap \mathcal{P}_{\alpha_j}(S_{0,n+1})$ 
and let $\widehat{e_j}$ be the thick edge in $\sth_{X_{n+1}}(u) \cap \mathcal{P}_{\alpha_j}(S_{0,n+1})$ but not in $\sth_{X_{n+1}}(u)\cap \mathcal{P}_{\alpha_i}(S_{0,n+1})$. 
Since $Q_i$ is a deficiency-$(n-3)$ multicurve, there are $n-3$ Farey graphs in $\mathcal{P}_{Q_i}(S_{0,m})=\mathcal{P}_{Q_j}(S_{0,m})$ that contain $f^{Q_i}_i(u)=f^{Q_j}_j(u)$.
However, $f_i^{Q_i}(\widehat e_i)$, $f_j^{Q_j}(\widehat e_j)$ and the $n-4$ thick edges in $\sth_{X_{n+1}}(u) \cap \mathcal{P}_{\alpha_i}(S_{0,n+1}) \cap \mathcal{P}_{\alpha_j}(S_{0,n+1})$ above all map into distinct Farey graphs by Lemma 5.2.  Since this is $n-2 > n-3$, we obtain a contradiction.
\end{proof}
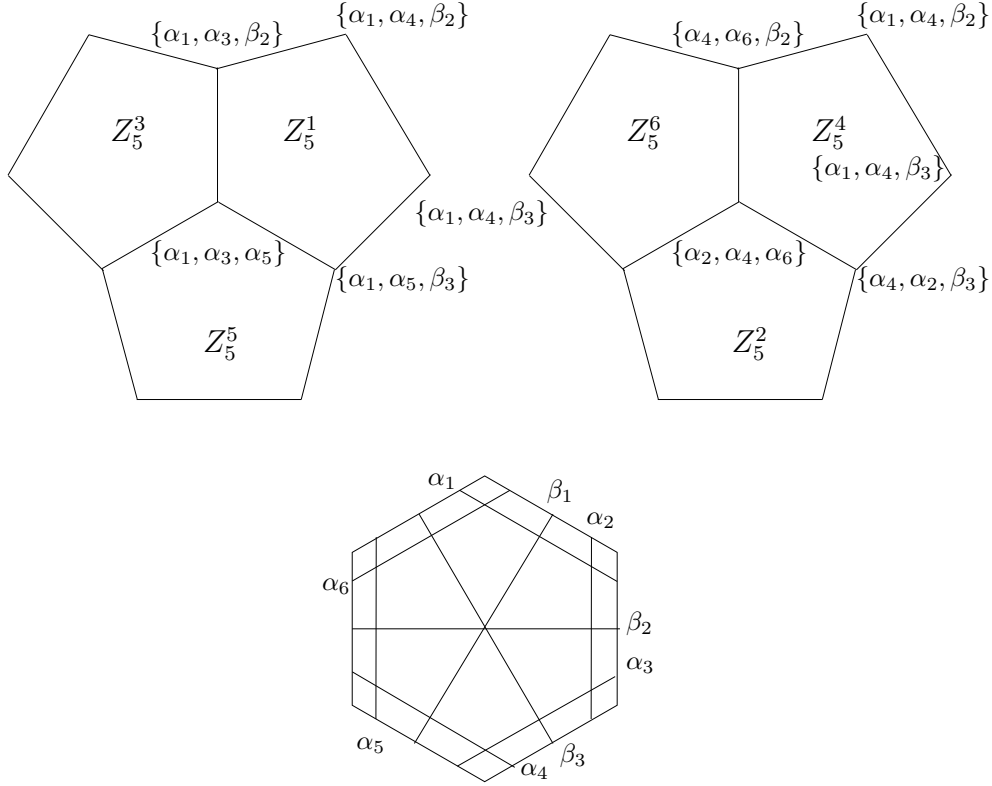
\begin{figure}[ht]
\begin{center}
\includegraphics[height=10cm]{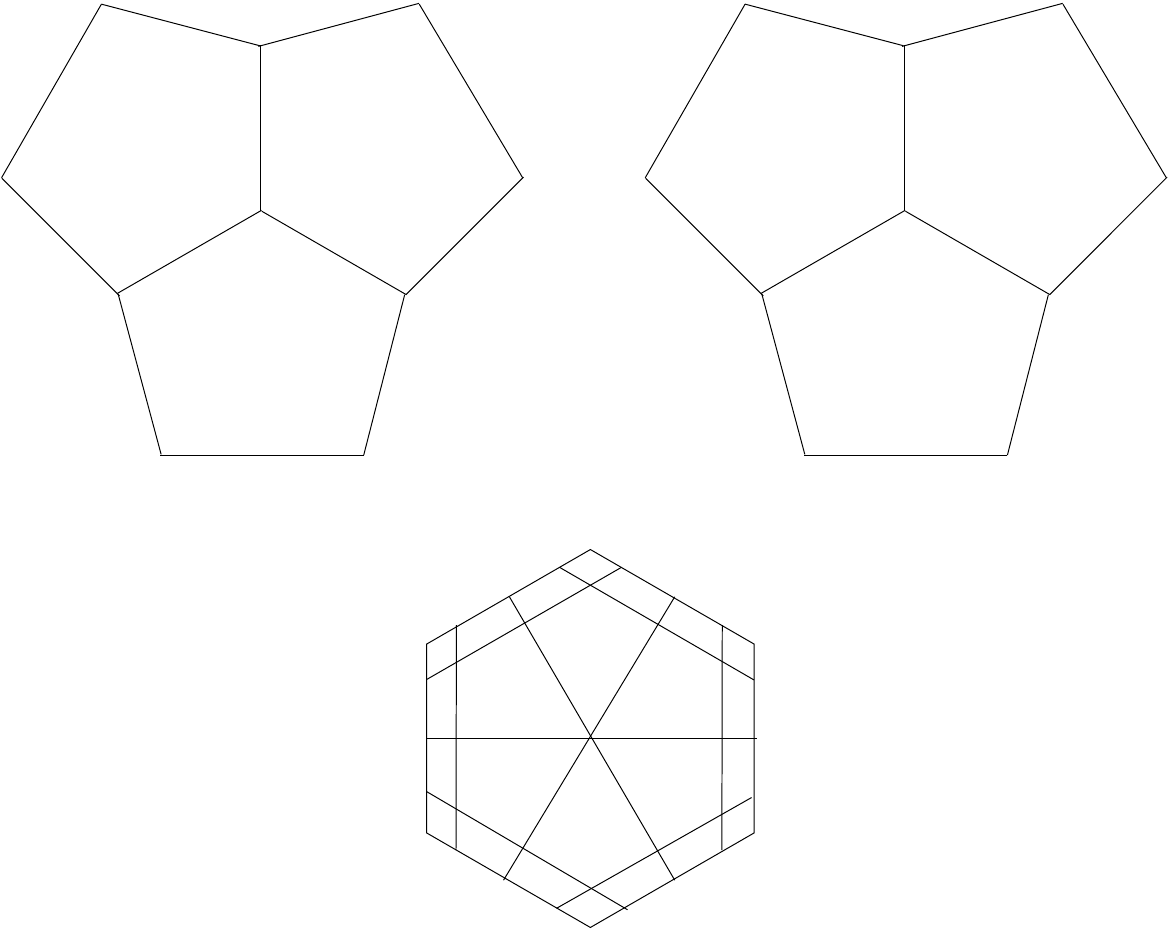} 
\caption{(Upper) $Z_6$ as a union of $Z^1_5\cup Z^3_5\cup Z^5_5$ on the left and $Z^2_5\cup Z^4_5\cup Z^6_5$ on the right, the two subgraphs share three edges as labelled; (lower) $S_{0,6}$ with curves in $\Gamma_6$.} 
\label{F:X6}
\end{center}
\begin{picture}(0,0)(0,0)

\put(70,358){\footnotesize $\left\{\alpha_1,\alpha_3,\beta_2\right\}$}
\put(140,365){\footnotesize $\left\{\alpha_1,\alpha_4,\beta_2\right\}$}
\put(170,290){\footnotesize $\left\{\alpha_1,\alpha_4,\beta_3\right\}$}
\put(140,265){\footnotesize $\left\{\alpha_1,\alpha_5,\beta_3\right\}$}
\put(70,275){\footnotesize $\left\{\alpha_1,\alpha_3,\alpha_5\right\}$}

\put(267,358){\footnotesize $\left\{\alpha_4,\alpha_6,\beta_2\right\}$}
\put(337,365){\footnotesize $\left\{\alpha_1,\alpha_4,\beta_2\right\}$}
\put(320,307){\footnotesize $\left\{\alpha_1,\alpha_4,\beta_3\right\}$}
\put(337,265){\footnotesize $\left\{\alpha_4,\alpha_2,\beta_3\right\}$}
\put(267,275){\footnotesize $\left\{\alpha_2,\alpha_4,\alpha_6\right\}$}

\put(175,190){\footnotesize $\alpha_1$}
\put(235,175){\footnotesize $\alpha_2$}
\put(250,120){\footnotesize $\alpha_3$}
\put(210,80){\footnotesize $\alpha_4$}
\put(148,90){\footnotesize $\alpha_5$}
\put(135,150){\footnotesize $\alpha_6$}
\put(220,185){\footnotesize $\beta_1$}
\put(250,135){\footnotesize $\beta_2$}
\put(225,85){\footnotesize $\beta_3$}

\put(120,320){$Z^1_5$}
\put(55,320){$Z^3_5$}
\put(90,240){$Z^5_5$}
\put(250,320){$Z^6_5$}
\put(320,320){$Z^4_5$}
\put(290,240){$Z^2_5$}

\end{picture}
\end{figure}
With the hypothesis as in Lemma~\ref{lem:3}, we must have
\[
Q_i\cap Q_j=Q_i\cap Q_k=Q_j\cap Q_k=Q_i\cap Q_j\cap Q_k=Q,
\]
 $Q$ has deficiency $n-2$ and there is a connected complexity $n-2$ component $(S_{0,m}-Q)_0$.

The proofs are different for $n+1 = 6$ and $n+1 \geq 7$. We first prove the theorem for $S_{0,n+1}=S_{0,6}$. We label the six chain curves by $\alpha_i$ for $i=1,...,6$ and the three non-chain curves in $\Gamma_6$ by $\beta_i$ for $i=1,2,3$ as in Figure~\ref{F:X6}; also see the figure for the picture of $Z_6$ as a union of $Z^1_5\cup Z^3_5\cup Z^5_5$ and $Z^2_5\cup Z^4_5\cup Z^6_5$, where $Z^i_5=Z_6\cap \mathcal{P}_{\alpha_i}(S_{0,6})=X^i_5\cap Z_6$. Note that, in this case, $Q_i$ is a deficiency-$2$ multicurve for all $i=1,...,6$.

Let $i\neq j\in \{1,3,5\}$ or $i\neq j\in \{2,4,6\}$. Define $F_{i,j}:S_{0,6}\to S_{0,m}$ by
\begin{displaymath}
   F_{i,j} = \left\{
     \begin{array}{lr}
       f_i & \mbox{on}~~(S_{0,6}-\alpha_i)_0\\
       f_j & \mbox{on}~~(S_{0,6}-\alpha_j)_0\\
     \end{array}
   \right. .
\end{displaymath} 

We claim that $f_i$ and $f_j$ agree on $(S_{0,6}-\left\{\alpha_i,\alpha_j\right\})_0$ so $F_{i,j}$ is well-defined. We prove the claim in the case when $(i,j)=(1,5)$. For other cases the proof is similar.

$X^1_5$ and $X^5_5$ share one thick edge $\hat{e}$ having endpoints $\left\{\alpha_1,\alpha_5,\alpha_3\right\}$ and $\left\{\alpha_1,\alpha_5,\beta_3\right\}$. 
Moreover, $f^{Q_1}_1$ agrees with $f^{Q_5}_5$ on $\hat{e}$ as they are both equal to $\phi$.
Then we have that $f^{Q_1}_1$ and $f^{Q_5}_5$ agree on $\mathcal{P}_{\alpha_1}(S_{0,6})\cap \mathcal{P}_{\alpha_5}(S_{0,6})=\mathcal{P}_{\left\{\alpha_1,\alpha_5\right\}}(S_{0,6})\cong \mathcal{P}(S_{0,4})$. 
Hence 
\[
f_1((S_{0,6}-\left\{\alpha_1,\alpha_5\right\})_0)=f_5((S_{0,6}-\left\{\alpha_1,\alpha_5\right\})_0),
\]
and $f^{-1}_5\circ f_1$ is either the identity or one of the three hyperelliptic involutions on $(S_{0,6}-\left\{\alpha_1,\alpha_5\right\})_0$, see Figure~\ref{F:hyperelliptic}. We will show that the latter are not possible by proving that $f^{-1}_5\circ f_1$ is the identity permutation on the boundary component of $(S_{0,6}-{\alpha_1,\alpha_5})_0$. 
Since $\alpha_1,\alpha_3,\alpha_5$ are pairwise disjoint chain curves, Lemma~\ref{lem:1} and Lemma~\ref{lem:3} show that, for any $i\neq j\in\left\{1,3,5\right\}$, $i(Q_i,Q_j)=0$ and $Q_i\neq Q_j$ . 
Hence, for any $i\neq j\in\left\{1,3,5\right\}$, the deficiency of $Q_i\cup Q_j$ is $1$ and the symmetric difference $Q_i\triangle Q_j$ contains two simple closed curves.
Let
\[
q_1\in Q_1-(Q_3\cup Q_5),~~q_5\in Q_5-(Q_1\cup Q_3),~~q_3\in Q_3-(Q_1\cup Q_5),
\]
which are three distinct simple closed curves on $S_{0,m}$. We note that $q_3$ is a closed curve on $(S_{0,m}-(Q_1\cup Q_5))_0\cong S_{0,4}$ and $q_3$ cannot separate $q_1$ and $q_5$ since $q_1$ and $q_5$ are curves on $(S_{0,m}-Q_3)_0\cong S_{0,5}$; see Figure~\ref{F:image}.
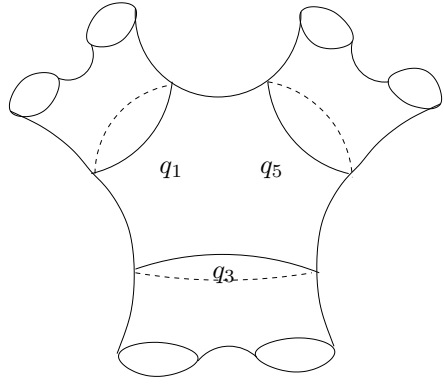
\begin{figure}[ht]
\begin{center}
\includegraphics[height=5cm]{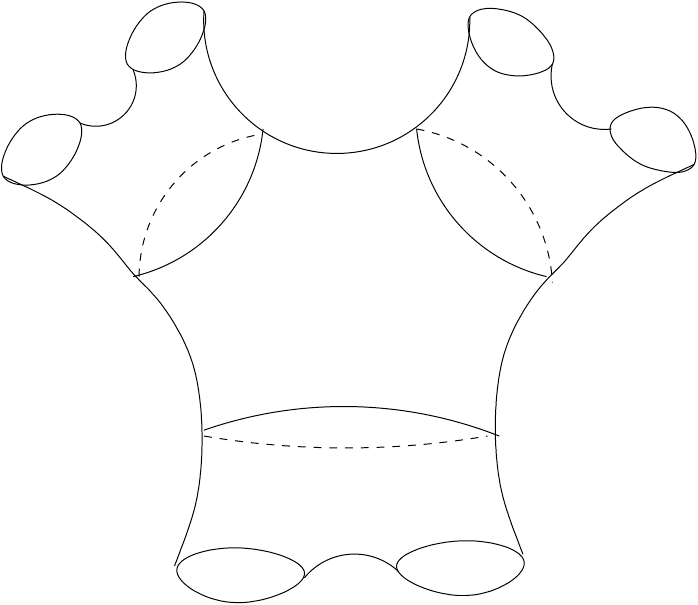} 
\caption{$f_1((S_{0,6}-\alpha_1)_0)\cup f_5((S_{0,6}-\alpha_5)_0)=(S_{0,m}-Q)_0$, together with the three curves $q_1,q_3,q_5$.} 
\label{F:image}
\end{center}
\begin{picture}(0,0)(0,0)

\put(170,140){\footnotesize $q_1$}
\put(208,140){\footnotesize $q_5$}
\put(190,100){\footnotesize $q_3$}

\end{picture}
\end{figure}

Since $f^{Q_1}_1$ and $f^{Q_5}_5$ agree on $\mathcal{P}_{\left\{\alpha_1,\alpha_5\right\}}(S_{0,6})\cong \mathcal{P}(S_{0,4})$, $f_1$ and $f_5$ map any simple closed curves on $(S_{0,6}-\left\{\alpha_1,\alpha_5\right\})_0$ to simple closed curves on $(S_{0,m}-(Q_1\cup Q_5))_0$. 
Thus 
$f_1(\alpha_5)$ is a curve in $(S_{0,m}-(Q_1))_0$ but not in $(S_{0,m}-(Q_1 \cup Q_5))_0$, and so $f_1(\alpha_5) = q_5$.  Similarly, $f_5(\alpha_1) = q_1$ and
\[ f_1(\alpha_3) = q_3 = f_5(\alpha_3).\]
%
%
Therefore $f_1(\alpha_1)=q_1$ and $f^{-1}_5\circ f_1(\alpha_1)=\alpha_1$.
We conclude that $f^{-1}_5\circ f_1$ is the identity permutation on the boundary component of $(S_{0,6}-\left\{\alpha_1,\alpha_5\right\})_0$ hence cannot be one of the hyperelliptic involutions. So $f_1$ and $f_5$ agree on $(S_{0,6}-\left\{\alpha_1,\alpha_5\right\})_0$.

Observe that $F_{1,5}=F_{1,3}=F_{3,5}$ and $F_{2,4}=F_{2,6}=F_{4,6}$. For instance, to see that $F_{1,5}=F_{1,3}$, we note that $f_3$ agrees with $f_5$ on $(S_{0,6}-\left\{\alpha_3,\alpha_5\right\})_0$. 

Define the common maps \[
f_{odd}=F_{1,5}=F_{1,3}=F_{3,5} ~~\mbox{and}~~f_{even}=F_{2,4}=F_{2,6}=F_{4,6}.
\]
We have that $f_{odd}$ induces the restriction map of $\phi$ on $X^1_5\cup X^3_5\cup X^5_5$ and  $f_{even}$ induces the restriction map of $\phi$ on $X^2_5\cup X^4_5\cup X^6_5$. 

Let $Q_{odd}=Q_1\cap Q_3\cap Q_5$ and $Q_{even}=Q_2\cap Q_4\cap Q_6$.
We show that $ (S_{0,m}-Q_{odd})_0=(S_{0,m}-Q_{even})_0$. 
As observed after Lemma~\ref{lem:3}, $Q_{odd}$ and $Q_{even}$ are deficiency-$3$ multicurves and 
\[
f_{odd}(S_{0,6})=(S_{0,m}-Q_{odd})_0~~\mbox{and}~~f_{even}(S_{0,6})=(S_{0,m}-Q_{even})_0.
\] 
Since $f^{Q_i}_i$ and $f^{Q_j}_j$ agree on the thick edge $\hat{e}_{i,j}\in X^i_5\cap X^j_5$ when $(i,j)\in\left\{(1,4),(3,6),(2,5)\right\}$, 
\[
f^{Q_{odd}}_{odd}(\hat{e}_{i,j})=f^{Q_{even}}_{even}(\hat{e}_{i,j})\in \mathcal{P}_{Q_{odd}}(S_{0,m})\cap \mathcal{P}_{Q_{even}}(S_{0,m})\neq\emptyset.
\]
Hence $i(Q_{odd},Q_{even})=0$. Moreover we have that $f_{odd}(\beta_k)=f_{even}(\beta_k)$, for each non-chain curve $\beta_k, k=1,2,3$ since $\beta_k$'s are interchanged in an elementary move that defines the edge $e_{i,j}$. 
The union $\bigcup^3_{k=1}f_{odd}(\beta_k)=\bigcup^3_{k=1}f_{even}(\beta_k)$ fills $(S_{0,m}-Q_{odd})_0$ and $(S_{0,m}-Q_{even})_0$.
Since $(S_{0,m}-Q_{odd})_0$ is the unique subsurface filled by $\bigcup^3_{k=1}f_{odd}(\beta_k)$, we conclude that $(S_{0,m}-Q_{odd})_0=(S_{0,m}-Q_{even})_0$.

Next, we show that $f_{odd}=f_{even}$.
Recall that we are considering $S_{0,6}$ as the double of a hexagon with vertices removed.
Let $r$ and $e$ be homeomorphisms of $S_{0,6}$ induced by rotating the hexagon by $\pi$ and exchanging the hexagons, respectively.
Since $f^{-1}_{odd}\circ f_{even}$ preserves each non-chain curve $\beta_k, k=1, 2, 3$, on $S_{0,6}$, $f^{-1}_{odd}\circ f_{even}$ is either the identity map, $r$, $e$, or $r\circ e$. 
We will show that only the first case is possible. 
The homeomorphism $r$ induces a simplicial map on $P(S_{0,6})$ which exchanges the vertices $\left\{\alpha_1,\alpha_3,\alpha_5\right\}$ and $\left\{\alpha_2,\alpha_4,\alpha_6\right\}$. 
If $f^{-1}_{odd}\circ f_{even}=r$ then $\phi(\left\{\alpha_1,\alpha_3,\alpha_5\right\})=\phi(\left\{\alpha_2,\alpha_4,\alpha_6\right\})$  which is a contradiction to the fact that $\phi$ is injective. 
Hence $f^{-1}_{odd}\circ f_{even}\neq r$.
Since $e$ and $r\circ e$ reverse orientation on each $\beta_i, i=1,2,3$, they do not induce the identity map on the thick edge $\hat{e}_{i,j}, (i,j)\in\left\{(1,4),(3,6),(2,5)\right\}$. So $f^{-1}_{odd}\circ f_{even}\neq e$ or $r\circ e$ because $f^{-1}_{odd}\circ f_{even}$ induces the identity map on those thick edges. Therefore $f^{-1}_{odd}\circ f_{even}$ is the identity map on $S_{0,6}$. We conclude that $f_{odd}=f_{even}$.

Let $f = f_{odd} = f_{even}$ and $Q = Q_{odd} = Q_{even}$.  We show that $f:S_{0,6} \to S_{0,m}$ is the unique $\pi_1$--injective embedding with $\phi = f^Q$.
The proof follows the same idea as in the proof of $f_{odd}=f_{even}$. So we give a brief explanation here.
Suppose $h:S_{0,6}\to S_{0,m}$ is a $\pi_1$-injective embedding that induces $h^W=\phi$ for some deficiency-$3$ multicurve $W$ on $S_{0,m}$. 
Since $f^Q=\phi=h^W$, $i(Q,W)=0$ moreover $f(\beta_i)=h(\beta_i)$ for each non-chain curve $\beta_i, i=1,2,3$. 
Since $(S_{0,m}-Q)_0$ and $(S_{0,m}-W)_0$ are unique subsurfaces filled by $\bigcup^3_{i=1}f(\beta_i)=\bigcup^3_{i=1}h(\beta_i)$, $(S_{0,m}-Q)_0=(S_{0,m}-W)_0$ and hence $W = Q$.
Since $h^{-1}\circ f$ preserves each non-chain curve on $S_{0,6}$ and induces the identity map on $X_6$, $h^{-1}\circ f$ is the identity map. Therefore $f=h$.


 Next we prove the theorem for $S_{0,n+1}, n+1\geq 7$.
\\We define an embedding of $S_{0,n+1}$ to $S_{0,m}$ that induces $\phi$. Let $\alpha_i$ and $\alpha_j$ be two disjoint chain curves on $S_{0,n+1}$. Define a homeomorphism $F_{ij}:S_{0,n+1}\to S_{0,m}$ by
\begin{equation}
\label{eq:1}
   F_{ij} = \left\{
     \begin{array}{lr}
       f_i & \mbox{on}~~(S_{0,n+1}-\alpha_i)_0\\
       f_j & \mbox{on}~~(S_{0,n+1}-\alpha_j)_0\\
     \end{array}
   \right. .
\end{equation} 
Note that, by Lemma~\ref{lem:1} and~\ref{lem:3}, $i(Q_i,Q_j)=0$ and $Q_i\neq Q_j$.
We show that $f_i$ agrees with $f_j$ on $(S_{0,n+1}-\left\{\alpha_i,\alpha_j\right\})_0\cong S_{0,n-1}$. Consider the restrictions of $f_i$ and $f_j$ as embeddings on $(S_{0,n+1}-\left\{\alpha_i,\alpha_j\right\})_0$. Let $X^{ij}_{n-1}=X_{n+1}\cap \mathcal{P}_{\left\{\alpha_i,\alpha_j\right\}}(S_{0,n+1})\cong X_{n-1}$. Observe that 
\[
f^{Q_i}_i(v)=\phi(v)=f^{Q_j}_j(v),
\]
 for any vertex $v\in X^{ij}_{n-1}$. That is, both $f_i$ and $f_j$ induce the simplicial map $\phi_{ij}: X^{ij}_{n-1} \to \mathcal{P}(S_{0,m})$ defined by $\phi_{ij}(v)=\phi(v)$. Then the uniqueness statement in the induction hypothesis (which applies since $n-1 \geq 5$)  implies that $f_i$ agrees with $f_j$ on $(S_{0,n+1}-\left\{\alpha_i,\alpha_j\right\})_0$. 
\begin{lem}
\label{lem:4}
For any four chain curves $\alpha_i, \alpha_j, \alpha_k$, and $\alpha_l$ such that $i(\alpha_i,\alpha_j)=0=i(\alpha_k,\alpha_l)$, we have $F_{ij}=F_{kl}$. 
\end{lem}
\begin{proof}
Consider a graph which has the set of vertices 
\[
V=\{\{i,j\}|\alpha_i, \alpha_j~\mbox{are disjoint chain curves on}~S_{0,n+1}\},
\] 
and the set of edges 
\[
E=\{\{\{i,j\},\{i,k\}\}|i(\alpha_j,\alpha_k)=0\}.
\] 
Fix any two vertices $\left\{i,j\right\}$ and $\left\{k,l\right\}$. Since $n+1\geq 7$, it is not hard to see that $\left\{i,j\right\}$ and $\left\{k,l\right\}$ are connected by a path of length at most $4$. Hence this graph is connected.

The Lemma now follows by proving that $F_{ij} = F_{ik}$ for $i(\alpha_i,\alpha_j) = 0 = i(\alpha_i,\alpha_k)$. We show that for any pair $(x,y), x\neq y\in \left\{i,j,k\right\}$, $f_x$ agrees with $f_y$ on $(S_{0,n+1}-\left\{\alpha_x,\alpha_y\right\})_0$. Since $(S_{0,n+1}-\left\{\alpha_x,\alpha_y\right\})_0$ has at least $5$ punctures and $f^{Q_x}_x(v)=f^{Q_y}_y(v)$ for any vertex $v\in X^{xy}_{n-1}$, the uniqueness statement in the induction hypothesis implies that $f_x$ agrees with $f_y$ on $(S_{0,n+1}-\left\{\alpha_x,\alpha_y\right\})_0$. Hence $F_{ij}=F_{ik}$ as desired.
\end{proof}
Finally, we show that for any $i,j$ with $i(\alpha_i,\alpha_j) = 0$, then $Q = Q_i \cap Q_j$ and  $f=F_{ij}:S_{0,n+1}\to S_{0,m}$ is the unique $\pi_1$-injective embedding that induces $\phi=f^Q$. 
To show that $f$ induces $\phi$, given a vertex $v\in X_{n+1}$. $v\in X^{pk}_{n-1}=X_{n+1}\cap \mathcal{P}_{\{\alpha_p,\alpha_k\}}(S_{0,n+1})$ for some disjoint chain curves $\alpha_p, \alpha_k$. Then $Q=Q_i \cap Q_j=Q_p \cap Q_k$ and
\[
f^Q(v)=F^Q_{pk}(v)=f^{Q_p}_p(v)=\phi(v),
\]
The first equality comes from Lemma~\ref{lem:4}, the second equality comes from equation~\ref{eq:1}, and the third equality comes from the inductive hypothesis.
Hence $f$ induces $\phi$.

For uniqueness, assume that there is a $\pi_1$-injective embedding $h:S_{0,n+1}\to S_{0,m}$ that induces $h^W=\phi$ for some deficiency-$n-2$ multicurve $W$ on $S_{0,m}$. 
Let $h_i$ and $h_j$ be restrictions of $h$ on $(S_{0,n+1}-\alpha_i)_0$ and $(S_{0,n+1}-\alpha_j)_0$, respectively. 
By assumption, $h_i$ induces the same simplicial map on $X^i_n$ as $f^{Q_i}_i$. 
Hence the uniqueness statement in the induction hypothesis implies that $f_i=h_i$ and $Q_i=W_i$. Similarly, we have $f_j=h_j$ and $Q_j=W_j$. Therefore $f=h$ and $Q=W$ as desired. 
\end{proof}


\begin{thebibliography}{1}

\bibitem{Aramayona} J. Aramayona, Simplicial embeddings between pants graphs, Geom. Dedicata. \textbf{144} (2010), no. 1, 115-128. MR2580421
\bibitem{AL} J. Aramayona and C. J. Leininger, Finite rigid sets in curve complexes, J. Topol. Anal. \textbf{5} (2013), no. 2, 183-203. MR3062946
\bibitem{APS} J. Aramayona, H. Parlier, and K. J. Shackleton, Totally geodesic subgraphs of the pants complex, Math. Res. Lett. \textbf{15} (2008), no. 2, 309-320. MR2385643
\bibitem{HT} A. Hatcher and W. Thurston, A presentation for the mapping class group of a closed orientable surface, Topology. \textbf{19} (1980), no. 3, 221-237. MR579573
\bibitem{Ivanov} N. V. Ivanov, Automorphisms of complexes of curves and of Teichm\"uller spaces, Internat. Math. Res. Notices. \textbf{1997} (1997), no. 14, 651-666. MR1460387
\bibitem{Korkmaz} M. Korkmaz, Automorphisms of complexes of curves on punctured spheres and on punctured tori, Topology Appl. \textbf{95} (1999), no. 2, 85-111. MR1696431
\bibitem{Luo} F. Luo, Automorphisms of the complex of curves, Topology. \textbf{39} (2000), no. 2, 283-298. MR1722024
\bibitem{Mar} D. Margalit, Automorphisms of the pants complex, Duke Math. J. \textbf{121} (2004), no. 3, 457-479. MR2040283
\bibitem{Minsky} Y. N. Minsky, A geometric approach to the complex of curves on a surface, in: S. Kojima et. al. (Eds.), Topology and {T}eichm\"uller spaces ({K}atinkulta, 1995), World Scientific, River Edge, New Jersey, 1996, pp. 149-158. MR1659683
\bibitem{SK} K. J. Shackleton, Combinatorial rigidity in curve complexes and mapping class groups, Pacific J. Math. \textbf{230} (2007), no. 1, 217-232 MR2318453
  
\end{thebibliography}
\end{document}